\documentclass[reqno,a4paper,12pt]{amsart}
\usepackage{amsmath,amsthm,amssymb,mathrsfs}
\usepackage[
	left = 2.7cm,
	right = 2.7cm,
	top = 2.2 cm,
	bottom = 2.2 cm
]{geometry}
\usepackage{graphicx,xcolor,tikz,mathtools}
\usepackage{bm,stmaryrd}
\usepackage{enumitem}
\usepackage[toc,page]{appendix}
\usepackage[
	pdfencoding=auto,
	colorlinks = true,
	citecolor = blue,
	linkcolor = blue,
	anchorcolor = blue,
	urlcolor = blue, 
	filecolor=blue,
	pdfkeywords={}
]{hyperref}

\theoremstyle{plain}
\newtheorem{theorem}{Theorem}[section]
\newtheorem{corollary}[theorem]{Corollary}

\newtheorem{lemma}[theorem]{Lemma}
\newtheorem{proposition}[theorem]{Proposition}
\newtheorem{remark}[theorem]{Remark}
\theoremstyle{definition}
\newtheorem{assumption}{Assumption}

\numberwithin{equation}{section}

\newcommand{\bbr}{\mathbb{R}}
\newcommand{\bbi}{\mathbb{I}}
\newcommand{\bbn}{\mathbb{N}}
\newcommand{\bbt}{\mathbb{T}}
\newcommand{\bbp}{\mathbb{P}}

\newcommand{\bv}{\mathbf{v}}
\newcommand{\vf}{\bv_{f}}
\newcommand{\vs}{\bv_{s}}
\newcommand{\bu}{\mathbf{u}}
\newcommand{\uf}{\bu_{f}}
\newcommand{\us}{\bu_{s}}
\newcommand{\hv}{\hat{\bv}}
\newcommand{\hu}{\hat{\bu}}
\newcommand{\hvf}{\hv_{f}}
\newcommand{\hvs}{\hv_{s}}

\newcommand{\hus}{\hu_{s}}

\newcommand{\bF}{\mathbf{F}}

\newcommand{\Fse}{\bF_{s,e}}

\newcommand{\hF}{\hat{\bF}}
\newcommand{\hFf}{\hF_{f}}
\newcommand{\hFs}{\hF_{s}}
\newcommand{\hFse}{\hF_{s,e}}
\newcommand{\hFsg}{\hF_{s,g}}
\newcommand{\Oft}{\Omega_{f}^{t}}
\newcommand{\Ost}{\Omega_{s}^{t}}

\newcommand{\Gt}{\Gamma^{t}}
\newcommand{\Of}{\Omega_{f}}
\newcommand{\Os}{\Omega_{s}}

\newcommand{\Gs}{\Gamma_s}
\newcommand{\Gst}{\Gamma_s^t}
\newcommand{\OM}{\Omega}

\newcommand{\hnab}{\hat{\nabla}}

\newcommand{\Jse}{J_{s,e}}

\newcommand{\hJ}{\hat{J}}
\newcommand{\hJf}{\hJ_f}
\newcommand{\hJs}{\hJ_s}
\newcommand{\hJse}{\hJ_{s,e}}
\newcommand{\hJsg}{\hJ_{s,g}}
\newcommand{\hg}{\hat{g}}
\newcommand{\rf}{\rho_f}
\newcommand{\rs}{\rho_s}
\newcommand{\hr}{\hat{\rho}}
\newcommand{\hrf}{\hr_f}
\newcommand{\hrs}{\hr_s}
\newcommand{\hc}{\hat{c}}
\newcommand{\hcf}{\hc_f}
\newcommand{\hcs}{\hc_s}
\newcommand{\hcss}{\hcs^*}
\newcommand{\nuf}{\nu_f}
\newcommand{\nus}{\nu_s}

\newcommand{\vp}{\varphi}
\newcommand{\cL}{\mathcal{L}}

\newcommand{\wo}{\bw_0}
\newcommand{\cA}{\mathcal{A}}
\newcommand{\cR}{\mathcal{R}}
\newcommand{\cF}{\mathcal{F}}

\newcommand{\bn}{\mathbf{n}}
\newcommand{\hn}{\hat{\bn}}
\newcommand{\ngt}{\bn_{\Gt}}
\newcommand{\ngst}{\bn_{\Gst}}
\newcommand{\hng}{\hn_{\Gamma}}
\newcommand{\hngs}{\hn_{\Gs}}

\newcommand{\hpi}{\hat{\pi}}
\newcommand{\pif}{\pi_f}
\newcommand{\pis}{\pi_s}
\newcommand{\hpif}{\hpi_f}
\newcommand{\hpis}{\hpi_s}

\newcommand{\cD}{\mathcal{D}}
\newcommand{\Dq}{\cD_q}
\newcommand{\Dqv}{\Dq^1}
\newcommand{\Dqc}{\Dq^2}
\newcommand{\YT}{Y_T}
\newcommand{\ZT}{Z_T}
\newcommand{\bS}{\mathbf{S}}

\newcommand{\bK}{\mathbf{K}}
\newcommand{\tk}{\tilde{\bK}}
\newcommand{\kf}{\tk_f}
\newcommand{\ks}{\tk_s}

\newcommand{\bH}{\mathbf{H}}
\newcommand{\bh}{\mathbf{h}}

\newcommand{\FfOinv}{\inv{\hFf} - \bbi}
\newcommand{\FfOinvtran}{\invtr{\hFf} - \bbi}

\newcommand{\FsOinv}{\inv{\hFs} - \bbi}
\newcommand{\FsOinvtran}{\invtr{\hFs} - \bbi}
\newcommand{\TD}{T^{\delta}}

\newcommand{\hD}{\hat{D}}
\newcommand{\hDf}{\hD_f}
\newcommand{\hDs}{\hD_s}
\newcommand{\vo}{\hv^0}
\newcommand{\co}{\hc^0}

\newcommand{\tF}{\tilde{F}}
\newcommand{\tFf}{\tF_f}
\newcommand{\tFs}{\tF_s}

\newcommand{\barpi}{\bar{\pi}}

\newcommand{\tpi}{\tilde{\pi}}

\newcommand{\bR}{\mathbf{R}}

\newcommand{\sL}{\mathscr{L}}
\newcommand{\sN}{\mathscr{N}}
\newcommand{\sM}{\mathscr{M}}
\newcommand{\cE}{\mathcal{E}}
\newcommand{\cM}{\mathcal{M}}
\newcommand{\bx}{\mathbf{x}}
\newcommand{\bX}{\mathbf{X}}
\newcommand{\bQ}{\mathbf{Q}}
\newcommand{\be}{\mathbf{e}}
\newcommand{\bw}{\mathbf{w}}

\newcommand{\hwf}{\hat{\bw}_f}
\newcommand{\hws}{\hat{\bw}_s}

\newcommand{\Smu}{S_{\mu}}
\newcommand{\ptial}[1]{ \partial_{#1} }

\renewcommand{\Bar}[1]{\overline{#1}}

\newcommand{\cP}{\mathcal{P}}

\newcommand{\bbe}{\mathbb{E}}

\newcommand{\bbf}{\mathbb{F}}

\newcommand{\delt}{\delta(T)}

\newcommand{\onehalf}{\frac{1}{2}}
\newcommand{\oneq}{\frac{1}{q}}
\newcommand{\onehalfq}{\frac{1}{2q}}
\newcommand{\onequater}{\frac{1}{4}}
\newcommand{\tr}{\mathrm{tr}}
\newcommand{\rv}[1]{\left. #1 \right\vert}
\newcommand{\rvm}[1]{ #1 \vert}
\newcommand{\tran}[1]{ #1^{\top}}
\newcommand{\inv}[1]{ #1^{-1}}
\newcommand{\invtr}[1]{ #1^{-\top}}
\renewcommand{\d}{\mathrm{d}}
\newcommand{\pt}{\partial_t}
\newcommand{\jump}[1]{\left\llbracket #1 \right\rrbracket}
\newcommand{\abs}[1]{\left\vert #1 \right \vert}

\newcommand{\norm}[1]{\left\Vert #1 \right \Vert}
\newcommand{\normm}[1]{\Vert #1 \Vert}
\newcommand{\seminorm}[1]{\left[ #1 \right]}

\newcommand{\inner}[2]{\left\langle #1 , #2 \right\rangle}

\newcommand{\Lq}[1]{L^{#1}}

\newcommand{\Lqs}[1]{L^{#1}_{\sigma}}
\newcommand{\W}[1]{W^{#1}_{q}}
\newcommand{\WO}[1]{{_0W}^{#1}_{q}}

\newcommand{\Bq}[1]{B^{#1}_{q,q}}

\newcommand{\Bqp}[1]{B^{#1}_{q,p}}
\newcommand{\tin}{\quad \text{in }}
\newcommand{\ton}{\quad \text{on }}
\renewcommand{\H}[1]{H^{#1}_{q}}
\newcommand{\HO}[1]{{_0H}^{#1}_{q}}

\newcommand{\K}[1]{K^{#1}_{q}}
\newcommand{\KO}[1]{{_0K}^{#1}_{q}}

\def\normdist{0.3ex}
\newcommand{\NORM}[1]{\left\vert\kern-\normdist\left\vert\kern-\normdist\left\vert #1 \right\vert\kern-\normdist\right\vert\kern-\normdist\right\vert}

\newcommand{\doi}[1]{DOI: \href{https://doi.org/#1}{#1}}
\newcommand{\arxiv}[1]{arXiv: \href{https://arxiv.org/abs/#1}{#1}}

\DeclareMathOperator*{\Div}{\mathrm{div}}
\DeclareMathOperator*{\hdiv}{\widehat{\Div}}

\newcommand{\dist}{\mathrm{dist}}
\newcommand{\symm}{\mathrm{sym}}
\newcommand{\hDelta}{\widehat{\Delta}}
\allowdisplaybreaks[4]

\begin{document}
	
	\title[Quasi-stationary fluid-structure interaction problem]{Short time existence of a quasi-stationary fluid-structure interaction problem for plaque growth}
	
	
	\author{Helmut Abels}
	\address{Fakult\"at f\"ur Mathematik, Universit\"at Regensburg, 93053 Regensburg, Germany}
	\email{Helmut.Abels@ur.de}

	\author{Yadong Liu}
	\address{Fakult\"at f\"ur Mathematik, Universit\"at Regensburg, 93053 Regensburg, Germany}
	\email{Yadong.Liu@ur.de}
	\date{\today}
	
	
	\thanks{Y. Liu is supported by the RTG 2339 ``Interfaces, Complex Structures, and Singular Limits'' of the German Science Foundation (DFG). The support is gratefully acknowledged.}
	
	\begin{abstract}
		We address a quasi-stationary fluid-structure interaction problem coupled with cell reactions and growth, which comes from the plaque formation during the stage of the atherosclerotic lesion in human arteries. The blood is modeled by the incompressible Navier--Stokes equation, while the motion of vessels is captured by a quasi-stationary equation of nonlinear elasticity. The growth happens when both cells in fluid and solid react, diffuse and transport across the interface, resulting in the accumulation of foam cells, which are exactly seen as the plaques. Via a fixed-point argument, we derive the local well-posedness of the nonlinear system, which is sustained by the analysis of decoupled linear systems.
	\end{abstract}
	
	\maketitle
	
%
	
	
\section{Introduction}
	\label{sec:introduction}
	\subsection{Model description}
	\label{sec:model-description}
	In this article, we consider a quasi-stationary fluid-structure interaction problem for plaque growth, which describes the formation of plaque during the reaction-diffusion and transport of different cells in human blood and vessels.
	The problem is set up in a smooth domain $ \OM^t \subset \bbr^3 $, with three disjoint parts $ \OM^t = \Oft \cup \Ost \cup \Gt $, where $ \Gamma^t = \partial \Oft $, $ \Bar{\Oft} \subset \OM^t $ and $ \Oft $, $ \Ost $ denote the domains for the fluid and solid, respectively. $ \Gst = \partial \Omega^t $ stands for the outer boundary of $ \OM^t $, which is also a free boundary. 
	Given $ T > 0 $, let
	\begin{alignat*}{4}
		Q^T & := \bigcup_{t \in (0,T)} \Omega^t \times \{ t \}, & \quad 
		Q_{f/s}^T & := \bigcup_{t \in (0,T)} \Omega_{f/s}^t \times \{ t \}, \\
		S^T & := \bigcup_{t \in (0,T)} \Gamma^t \times \{ t \}, & \quad 
		S_s^T & := \bigcup_{t \in (0,T)} \Gamma_s^t \times \{ t \}.
	\end{alignat*}
	As in \cite{AL2021a,AL2021b,Yang2016}, the flood is described by the incompressible Navier--Stokes equation
	\begin{alignat}{3}
			\rho_f (\pt + \bv_f \cdot \nabla) \bv_f & = \Div \bbt_f, && \tin Q_f^T, \label{Eqs:v_f-Eulerian}\\
			\Div \bv_f & = 0, && \tin Q_f^T, \label{Eqs:fluid-mass-Eulerian}
	\end{alignat}
	where $ \bbt_f := - \pi_f \bbi + \nu_s (\nabla \bv_f + \nabla^\top \bv_f) $ denotes the Cauchy stress tensor. $ \vf : \bbr^3 \times \bbr_+ \rightarrow \bbr^3 $, $ \pif : \bbr^3 \times \bbr_+ \rightarrow \bbr $ are the unknown velocity and pressure of the fluid. $ \rf > 0 $ stands for the fluid density and $ \nus $ represents the viscosity of the fluid.
	
	Compared to the problems in \cite{AL2021a,AL2021b,Yang2016}, where the evolutionary neo-Hookean material was employed, the vessel in this manuscript is assumed to be quasi-stationary since it moves far slower than the blood from a macro point of view. Thus, we model the blood vessel by the equilibrium of a nonlinear elastic equation. To mathematically describe the elasticity conveniently, the Lagrangian coordinate was commonly used, see e.g. \cite{Goriely2017:growth,Gurtin2010}. 
	Thus, we set reference configuration as the initial domain which is defined by $ \OM := \rvm{\OM^t}_{t = 0} $, as well as $ \Of = \Omega_f^0 $, $ \Os = \Omega_s^0 $ and $ \Gamma = \Gamma^0 $, then $ \OM = \Of \cup \Os \cup \Gamma $. Let $ \bX $ be the spatial variable in the reference configuration. Now we introduce the Lagrangian flow map
	\begin{equation*}
		\vp: \OM \times (0,T) \rightarrow Q_T,
	\end{equation*}
	with 
	\begin{equation}
		\label{Eqs:Lagrangian flow map}
		\bx(\bX, t) = \vp(\bX,t) = \bX + \bu(\bx(\bX, t), t)
	\end{equation}
	for all $ \bX \in \OM $ and $ \bx(\bX, 0) = \bX $, where 
	\begin{equation*}
		\bu(\bx(\bX, t), t)
		= \left\{
			\begin{aligned}
				& \uf(\bx(\bX, t),t) = \int_{0}^t \vf(\bx(\bX, t), \tau) \d \tau, && \text{if } \bX \in \Of, \\
				& \us(\bx(\bX, t),t), && \text{if } \bX \in \Os,
			\end{aligned}
		\right.
	\end{equation*}
	denotes the displacement for fluid or solid. In the sequel, without special statement, the quantities with a hat will indicate those in Lagrangian reference configuration, e.g., $ \hat{\bu}(\bX, t) = \bu(\bx(\bX, t), t) $, while the operators with a hat means those act on the quantities in Lagrangian coordinate. Then the tensor field
	\begin{equation}
		\label{Eqs:deformation gradient}
		\bF(\bx(\bX,t),t) = \hF(\bX,t) := \frac{\partial}{\partial \bX} \vp(\bX,t) = \hnab \vp(\bX,t) = \bbi + \hnab \hu(\bX, t), \  \forall\, \bX \in \OM,
	\end{equation}
	with $ \hFf(\bX, 0) = \bbi $ and $ \hFs(\bX, 0) = \bbi + \hnab \hus^0 $ is referred to be the deformation gradient and $ J = \hJ := \det(\hF) $ denotes its determinant. 
	For the blood vessels, since the growth is taken into account, we impose the so-called \textit{multiplicative decomposition} for the solid deformation gradient $ \hFs $ as
	\begin{equation*}
		\hFs = \hFse\hFsg
	\end{equation*}
	with $ \hJs = \hJse \hJsg $, where $ \hFse $ is the pure elastic deformation tensor and $ \hFsg $ denotes the growth tensor, which will be specified later. For more details about the decomposition, see e.g. \cite{Goriely2017:growth,JC2012,RHM1994,Yang2016}.
	Inspired by Goriely \cite[Chapter 11--13]{Goriely2017:growth} and Jones--Chapman \cite[Section 3.2]{JC2012}, a general incompressible hyperelastic material is considered for solid as
	\begin{equation}
		\label{Eqs:elastic-Eulerian}
		- \Div \bbt_s = 0, \tin Q_s^T, 
	\end{equation}
	where $	\bbt_s := - \pis \bbi + \Jse^{-1} DW(\Fse) \tran{\Fse} $ stands for the Cauchy stress tensor. $ \pis : \bbr^3 \times \bbr_+ \rightarrow \bbr $ is the unknown pressure of the solid. The scalar function $ W : \bbr^{3 \times 3} \rightarrow \bbr_+ $ is called the strain energy density function (also known as the stored energy density), which needs some general assumptions for the sake of analysis. 
	\begin{assumption}
		\label{assumption:W-energy-density}
		$  $ \par
		\begin{enumerate}[label=\textbf{(H\arabic*)}] 
			\item \label{assumptions:frame-difference} $ W $ is frame-indifferent, i.e., $ W(\bR \bF) = W(\bF) $, for all $ \bR \in SO(3) $ and $ \bF \in \bbr^{3 \times 3} $, where $ SO(3) := \{\mathbf{A} \in \bbr^{3 \times 3}: \tran{\mathbf{A}} \mathbf{A} = \bbi, \det \mathbf{A} = 1\} $ is the set of all proper orthogonal tensors.
			\item \label{assumptions:smoothness} $ W \in C^4(\bbr^{3 \times 3}; \bbr) $.
			\item \label{assumptions:DW(I)} $ DW(\bbi) = 0 $, $ W(\bR) = 0 $ for all $ \bR \in SO(3) $.
			\item \label{assumptions:coercive} There exists a constant $ C_0 > 0 $, such that $ W(\bF) \geq C_0 \dist^2(\bF, SO(3)) $.
		\end{enumerate}
		Here, $ DW(\bF) := \frac{\partial W}{\partial F_{ij}} \be_i \otimes \be_j $ for all $ \bF \in \bbr^{3 \times 3} $. $ \dist(\bF, SO(3)) := \min_{\bQ \in SO(d)} \abs{\bF - \bQ} $.
	\end{assumption}
	\begin{remark}
		\label{remark:epllipticity}
		In fact, the Assumption \ref{assumptions:coercive} implies that
		\begin{equation*}
			D^2 W(\bbi)\bF : \bF \geq C_1 \abs{\symm \bF}^2,
		\end{equation*}
		 for some constant $ C_1 > 0 $, where $ \symm \bF:= \onehalf(\bF + \tran{\bF}) $. Then for any $ \mathbf{a}, \mathbf{b} \in \bbr^3 $, one can derive the so-called \textit{Legendre-Hadamard} condition 
		 \begin{equation*}
			 D^2 W(\bbi)(\mathbf{a} \otimes \mathbf{b}) : (\mathbf{a} \otimes \mathbf{b}) \geq \frac{C_1}{2} \abs{\mathbf{a}}^2 \abs{\mathbf{b}}^2,
		 \end{equation*}
		 by the Taylor expansion and the polar decomposition, which ensures that the operator $ - \Div D^2 W(\bbi) \nabla $ is strongly normally elliptic, see e.g. \cite[Page 271]{PS2016}.
	\end{remark}
	
	Besides \eqref{Eqs:elastic-Eulerian}, the balance of mass should hold as well together with the growth, namely,
	\begin{equation}
		\label{Eqs:solid-mass-Eulerian}
		\left(\pt + \vs \cdot \nabla \right) \rs + \rs \Div \vs = f_s^g, \tin Q_s^T,
	\end{equation}
	where $ \rs > 0 $ is the solid density. $ \vs := \pt \us $ denotes the velocity of the solid. The function $ f_s^g := \gamma f_s^r $ with $ \gamma > 0 $ on the right-hand side of \eqref{Eqs:solid-mass-Eulerian} stands for the growth rate, which comes from the cells reaction assigned later. 
	
	There are two essential assumptions when either the density or the volume does not change with growth. Generally, a \textit{constant-density} type growth is assumed for the incompressible tissue, see e.g. \cite{JC2012,RHM1994}, in this paper we assume this kind of growth as in \cite{AL2021a,AL2021b,Yang2016}. Then by \cite[(13.10)]{Goriely2017:growth}, one obtains
	\begin{equation}
		\label{Eqs:grwoth-before}
		\hrs \tr(\inv{\hFsg} \pt \hFsg) = f_s^g, \tin \Os \times (0,T).
	\end{equation}

	As in \cite{AL2021a,AL2021b,Yang2016}, the plaque grows when the macrophages in the vessels accumulate a lot from the monocytes in the blood flow and turn to be the foam cells. Thus, denoting by $ c_f $, $ c_s $, $ c_s^* $ the monocytes, the macrophages and the foam cells respectively, we introduce
	\begin{alignat}{3}
		\pt c_f + \Div \left( c_f \vf \right) - D_f \Delta c_f & = 0, && \tin Q_f^T, \label{Eqs:c_f-Euler}\\
		\pt c_s + \Div \left( c_s \vs \right) - D_s \Delta c_s & = - f_s^r, && \tin Q_s^T, \label{Eqs:c_s-Euler}\\
		\pt c_s^* + \Div \left( c_s^* \vs \right) & = f_s^r, && \tin Q_s^T, \label{Eqs:c_ss-Euler}
	\end{alignat}
	where $ D_{f/s} > 0 $ are the diffusion coefficients in the blood and the vessel, respectively. $ f_s^r := \beta c_s $ with $ \beta > 0 $ stands for the reaction function, modeling the rate of transformation from macrophages into foam cells. Noticing that the foam cells is supposed to only move with vessels motion, one has \eqref{Eqs:c_ss-Euler} above.
	
	To close the system, one still needs to impose suitable boundary and initial conditions. For this purpose, we follow a similar setting as in authors' previous work \cite{AL2021a}. On the free interface $ \Gt $, one has the continuity of velocity and normal stress tensor for the fluid-structure part, while for the cells, the concentration flux is continuous and the jump of the concentration is determined by the permeability and flux across the vessel wall.
	\begin{alignat}{3}
		\label{Eqs:vjump-Eulerian}
		\jump{\bv} = 0, \quad \jump{\bbt} \ngt & = 0, && \ton S^T, \\
		\label{Eqs:cjump-Eulerian}
		\jump{D \nabla c} \cdot \ngt = 0, \quad \zeta \jump{c} - D_s \nabla c_s \cdot \ngt & = 0, && \ton S^T,
	\end{alignat}
	where $ \ngt $ represents the outer unit normal vertor on $ \Gt $ pointing from $ \Oft $ to $ \Ost $. For a quantity $ f $, $ \jump{f} $ denotes the jump defined on $ \Oft $ and $ \Ost $ across $ \Gt $, namely,
	\begin{equation*}
		\jump{f}(\bx) := \lim_{\theta \rightarrow 0} {f(\bx + \theta \ngt(\bx)) - f(\bx - \theta \ngt(\bx))}, \quad \forall\, \bx \in \Gt.
	\end{equation*} 
	$ \zeta $ denotes the permeability of the sharp interface $ \Gt $ with respect to the monocytes, which basically should depend on the hemodynamic stress $ \bbt_f \ngt $. It is supposed to be a constant for simplicity. Moreover, $ \Gst $ is assumed to be a free boundary as well due to the physical compatibility, see e.g. \cite[Remark 1.3]{AL2021a}. Then we give the boundary conditions on $ \Gst $,
	\begin{alignat}{3}
		\label{Eqs:vboundary-Eulerian}
		\bbt_s \ngst & = 0, && \ton S_s^T, \\
		\label{Eqs:cboundary-Eulerian}
		D_s \nabla c_s \cdot \ngst & = 0, && \ton S_s^T.
	\end{alignat}
	Finally, the initial values are prescribed as 
	\begin{alignat}{3}
		\label{Eqs:ofinitial-Eulerian}
		\rv{\vf}_{t = 0} = \vf^0, \quad \rv{c_f}_{t = 0} & = c_f^0, && \tin \Of, \\
		\label{Eqs:osinitial-Eulerian}
		\rv{\us}_{t = 0} = \us^0, \quad 
		\rv{c_s}_{t = 0} = c_s^0, \quad 
		\rv{c_s^*}_{t = 0} = c_*^0, \quad {\hFsg}\vert_{t = 0} & = g^0\bbi, && \tin \Os.
	\end{alignat}
	
	\subsection{Previous work}
	Our main goal is to investigate the short time existence of strong solutions to \eqref{Eqs:v_f-Eulerian}--\eqref{Eqs:osinitial-Eulerian}. Before discussing the technical details, let us recall some literature related to our work.
	
	In the case of 3d-3d fluid-structure interaction problem with a free interface, the strong solution results can be tracked back to Coutand--Shkoller \cite{CS2005}, who addressed the interaction problem between the incompressible Navier--Stokes equation and a linear Kirchhoff elastic material. The results were extended to fluid-elastodynamics case by them in \cite{CS2006}, where they regularized the hyperbolic elastic equation by a particular parabolic artificial viscosity and then obtained the existence of strong solutions by delicate a priori estimates. Thereafter, systems consisting of an incompressible Navier--Stokes equation and (damped) wave/Lam\'{e} equations were continuously investigated in e.g. \cite{IKLT2012,IKLT2017,KT2012,RV2014} with further developments. 
	Besides the models above, one refers to \cite{BG2010} for a compressible barotropic fluid coupled with elastic bodies and to \cite{SWY2021} for the magnetohydrodynamics (MHD)-structure interaction system, where the fluid is described by the incompressible viscous non-resistive MHD equation and the structure is modeled by the wave equation of a superconductor material.
	
	In the context of 3d-2d/2d-1d models, various kinds of models and results were established during the last twenty years. The widely focused case is the fluid-beam/plate systems where the beam/plate equations were imposed with different mechanical mechanism (rigidity, stretching, friction, rotation, etc.), readers are refer to e.g. \cite{CDEG2005,Grandmont2008,MC2013,TW2020} for weak solution results and to e.g. \cite{Beirao da Veiga2004,DS2020,GHL2019,Lequeurre2013,MT2021NARWA,Mitra2020} for strong solutions respectively. Besides, the fluid-shells interaction problems were studied as well in e.g. \cite{BS2018,LR2013} for weak solutions and in e.g. \cite{CCS2007,CS2010,MRR2020} for strong solutions. It is worth mentioning that in recent works \cite{DS2020,MT2021NARWA}, a maximal regularity framework, which requires lower initial regularity and less compatibility conditions compared to the energy method, was employed.
	
	Recently inspired by \cite{Yang2016}, the authors considered the local well-posedness of a fluid-structure interaction problem for plaque growth in a smooth domain \cite{AL2021a} and a cylindrical domain \cite{AL2021b} respectively. By the maximal regularity theory, one obtains the well-posedness of linearized systems and then derives the existence of a unique strong solution to the nonlinear system via a fixed-point argument. However, concerning the origin of the model, which consists of blood flow and blood vessels, it is reasonable that the movement of vessels is supposed to be far slower than the blood, namely, the time scales of motions of the fluid and the solid are distinguished and the kinetic energy of the vessel is neglected. Thus, in this paper, we take an incompressible quasi-stationary hyperelastic equation into account, i.e., for every time $ t $ the elastic stresses are in equilibrium. To the best of our knowledge, this is the first result concerning on the strong solutions to a quasi-stationary fluid-structure interaction problem with growth.
	
	\subsection{Technical discussions}
	Under the above setting, the fluid-structure part is of parabolic-elliptic type, while the cells part is similar to the ones in \cite{AL2021a,AL2021b}. To solve the nonlinear problem, our basic strategy is still a fixed-point argument in the framework of maximal regularity theory, while more issues come up when we consider the linearized systems. More precisely, for the linearization of the fluid-structure part, it is hard to solve it directly by the maximal regularity theory due to the parabolic-elliptic type coupling, which results in the unmatched regularity between $ \hvf $ and $ \hus $ on the sharp interface $ \Gamma $. To overcome the problem, one tries to decouple the system to a nonstationary Stokes equation with respect to fluid velocity $ \hvs $ and a quasi-stationary Stokes-type equation with regard to the solid displacement $ \hus $. Note that one key point is to separate the kinetic and dynamic condition on the interface correctly. Specifically, we impose the Neumann boundary condition for $ \hvs $ and a Dirichlet boundary condition for $ \hus $, see Section \ref{sec:analysis-linear} below. Otherwise if $ \rvm{\hvf}_{\Gamma} = \hvs $, one may face the problem that there is no regularity information about the velocity of solid $ \hvs = \ptial{t}\hus $, since the solid equation is quasi-stationary without any damping.
	
	Another issue is the choice of function spaces for the elastic equation, i.e., how to assemble suitable function spaces for the data in the linearized solid equation so that the regularity of $ \hnab \hus $ matches with $ \hnab \hvs $ of fluid on the interface in the nonlinear system. Our choice is 
	\begin{equation*}
		\mathbf{f} \in \H{1/2}(0,T; W_{q, \Gamma}^{-1}(\Os)^3) \cap \Lq{q}(0,T; \Lq{q}(\Os)^3).
	\end{equation*}
	This space is motivated by the observation that the nonstationary Stokes equation \eqref{Eqs:linear-nonsta} is uniquely solvable if the Neumann boundary data
	\begin{equation*}
		\bh \in \W{1/2 - 1/2q}(0,T; \Lq{q}(\Gamma)^3) \cap \Lq{q}(0,T; \W{1 - 1/q}(\Gamma)^3)
	\end{equation*}
	holds, which implies that
	\begin{equation*}
		D^2 W(\bbi) \hnab \hus \in \W{1/2 - 1/2q}(0,T; \Lq{q}(\Gamma)^3) \cap \Lq{q}(0,T; \W{1 - 1/q}(\Gamma)^3)
	\end{equation*}
	as well. Here $ H^s_q $ and $ W^s_q $ are Bessel potential space and Sobolev--Slobodeckij space respectively, which will be given later in Section \ref{sec:function-spaces}. In fact, the anisotropic Bessel potential space we assigned is a sharp regularity if one goes back to the anisotropic trace operator, see Lemma \ref{lemma:trace-time-regularity} below, and it is natural to equip $ \mathbf{f} $ with the regularity above. Because of this sharp regularity setting, one can not expect the Lipschitz estimates of nonlinear terms only with small time, which gives birth to an additional smallness assumption \eqref{Eqs:us0-smallness} on the initial solid displacement and pressure. Detailed discussion can be found in Remark \ref{remark:smallness} and Proposition \ref{prop:Lipschitzestimate} later. 
	
	To solve the quasi-stationary (linearized) elastic equation, one treated it as a Stokes-type problem with respect to the displacement $ \hus $ and the pressure $ \hpis $, due to the incompressibility. However, as we assigned the certain regularity space for it as above and the Stokes operator is not a standard one (namely, $ \Div(D^2 W(\bbi) \nabla \cdot ) $), one needs to consider a generalized stationary Stokes equation with $ \mathbf{f} $ in $ \Lq{q} $ and $ W_{q, \Gamma_1}^{-1} $ respectively, for which the maximal regularity of analytic C$ _0 $ semigroups is applied, as well as a complex interpolation method with a \textit{very weak solution} in $ \Lq{q} $ of a mixed-boundary Stokes-type equation, which can be solved by a duality argument.
	
	For the cells part there is a problem of the positivity for the concentrations, compared to \cite{AL2021a}. The idea in \cite{AL2021a} to prove it is to apply the maximum principle to the original equation and deduce a contradiction with the help of Hopf's Lemma. However, due to the lack of regularity of $ \vs = \pt \us $, one can not expect it to be H\"older continuous in space-time, even continuous. To deal with this trouble, we make use of the idea of mollification, i.e., approximating $ \vs $ by sufficient smooth functions $ \vs^\epsilon $ such that $ \int_0^t \vs^\epsilon \rightarrow \us $ in certain space. Then arguing by a similar procedure in \cite{AL2021a}, we obtain an approximate nonnegative solution $ c^\epsilon $. Finally one can show that it converges to a nonnegative function $ c $, which exactly satisfies the original equations of cell concentrations.
	
	\subsection{Structure of the paper}
	The paper is organized as follows. In Section \ref{sec:preliminaries} we briefly introduce some notations and function spaces together with some corresponding properties. Moreover, a reformulation of the system is done in Section \ref{sec:reformulation} and later we give the main result for the reformulated system.
	Section \ref{sec:analysis-linear} is devoted to three linearized systems in Section \ref{sec:nonsta-Stokes}, \ref{sec:sta-Stokes-mixed} and \ref{sec:heat-neumann} respectively. The main results of this section are the $ L^q $-solvability for these linear problems, for which a careful analysis is carried out.
	In Section \ref{sec:nonlinear-wpd}, we first introduce some preliminary lemmas in Section \ref{sec:useful-lemmas}, which will be frequently used in proving the Lipschitz estimates later in Section \ref{sec:lipschitz-estimates}. Then by the Banach fixed-point Theorem, we derive the short time existence of strong solutions to the nonlinear system in Section \ref{sec:nonlinear-proof}. Moreover, the cell concentrations are shown to be nonnegative, provided that the initial concentration is nonnegative.
	In addition, we 
	establish the solvability of a Stokes-type resolvent problem with mixed boundary condition in Appendix \ref{sec:sta-Stokes-lower-regularity}.
\section{Preliminaries}
	\label{sec:preliminaries}
	\subsection{Notations}
	\label{sec:notations}
	Let us introduce some notations. For a vector $ \bu \in \bbr^d $, $ d \in \bbn_+ $, we define the gradient as $ (\nabla \bu)_{ij} = \ptial{i} \bu_j $. For a matrix $ \mathbf{A} = (A_{ij})_{i,j=1}^d \in \bbr^{d \times d} $, we introduce the divergence of a matrix row by row as 
	$ (\Div \mathbf{A})_{i} = \sum_{j = 1}^{d} \ptial{j} A_{ij} $.
	Given matrices $ \mathbf{A} = (A_{ij})_{i,j=1}^d, \mathbf{B} = (B_{ij})_{i,j=1}^d \in \bbr^{d \times d} $, we know $ (\mathbf{A}\mathbf{B})_{ik} = \sum_{j = 1}^{d} A_{ij}B_{jk} $. Then we denote by $ \mathbf{A} : \mathbf{B} := \tr(\tran{\mathbf{B}}\mathbf{A}) = \sum_{j,k = 1}^{d}B_{jk} A_{jk} $ the Frobenius product and $ \abs{\mathbf{A}} := \sqrt{\mathbf{A} : \mathbf{A}} $ the induced modulus. For a differentiable and invertible mapping $ \mathbf{A} : I \subset \bbr \rightarrow \bbr^{d \times d} $, we have
	\begin{equation}
		\label{Eqs:dt:detA-invA}
		\frac{\d}{\d t} \det \mathbf{A} = \tr \Big(\inv{\mathbf{A}} \frac{\d}{\d t}\mathbf{A}\Big) \det \mathbf{A}, \quad\frac{\d}{\d t} \inv{\mathbf{A}} = - \inv{\mathbf{A}} \Big(\frac{\d}{\d t} \mathbf{A}\Big) \inv{\mathbf{A}}, \text{ for } t \in I,   
	\end{equation}
	see e.g. \cite{Goriely2017:growth,Gurtin2010}. Moreover, we define $ \bbr_{\symm+}^{d \times d} $ as the space of all positive-definite symmetric matrices and
	\begin{equation*}
		SO(d) := \{\mathbf{A} \in \bbr^{d \times d}: \tran{\mathbf{A}} \mathbf{A} = \bbi, \det \mathbf{A} = 1\}
	\end{equation*}
	as the set of all proper orthogonal tensors.
	
	Generally, $ B_X(x,r) $ denotes the open ball with radius $ r > 0 $ around $ x $ in a metric space $ X $. For normed spaces $ X, Y $ over $ \mathbb{K} = \bbr $ or $ \mathbb{C} $, the set of bounded, linear operators $ T : X \rightarrow Y $ is denoted by $ \cL(X,Y) $ and in particular, $ \cL(X) = \cL(X,X) $.
	
	As usual, the letter $ C $ in the present paper denotes a generic positive constant which may change its value from line to line, even in the same line, unless we give a special declaration.
	
	\subsection{Function spaces and properties}
	\label{sec:function-spaces}
	For an open set $ M \subset \bbr^d $ with Lipschitz boundary $ \partial M $, $ d \in \bbn_+ $, we recall the standard Lebesgue space $ \Lq{q}(M) $ and Sobolev space $ \W{m}(M) $, $ m \in \bbn $, $ 1 \leq q \leq \infty $, as well as 
	\begin{gather*}
		W_{q, 0}^{m}(M) = \overline{C_0^\infty(M)}^{W_q^{m}(M)}, \quad 
		W_q^{-m}(M) = \left[W_{q', 0}^{m}(M)\right]', 
	\end{gather*}
	where $ q' $ is the conjugate exponent of $ q $ satisfying $ 1/q + 1/q' = 1 $. Furthermore, we set the Sobolev spaces associated with $ \Gamma \subset \partial M $ as
	\begin{equation*}
		W_{q, \Gamma}^m(M) = \left\{ \psi \in \W{m}(M): \psi|_{\Gamma} = 0 \right\}, \quad
		W_{q, \Gamma}^{-m}(M) := [ W_{q', \Gamma}^m(M) ]',
	\end{equation*}
	The vector-valued variants are denoted by $ \Lq{q}(M; X) $ and $ \W{m}(M; X) $, where $ X $ is a Banach space. In particular, $ \Lq{q}(M) = \W{0}(M) $, $ \Lq{q}(M; X) = \W{0}(M; X) $.
	
	Moreover, for $ 1 < q < \infty $ and $ 1 \leq p \leq \infty $ the standard Besov space $ \Bqp{s} $ and Bessel potential space $ \H{s} $ coincide with the real and complex interpolation of Sobolev spaces respectively (see e.g. Lunardi \cite{Lunardi2018}, Runst--Sickel \cite{RS1996}, Triebel \cite{Triebel1978}) 
	\begin{equation*}
		\Bqp{s}(M) = \left( \W{k}(M), \W{m}(M) \right)_{\theta ,p}, \quad
		\H{s}(M) = \left[ \W{k}(M), \W{m}(M) \right]_{\theta},
	\end{equation*}
	where $ s = (1 - \theta)k + \theta m $, $ \theta \in (0,1) $, $ k, m \in \bbn $, $ k < m $. In particular, $ \W{l} = \H{l} $ for $ l \in \bbn $ and setting $ p = q $ above yields the Sobolev--Slobodeckij space
	\begin{equation*}
		\W{s}(M) = \Bq{s}(M) = \left( \W{k}(M), \W{m}(M) \right)_{\theta ,q}, \text{ if } s \notin \bbn.
	\end{equation*}
	Then we define the seminorm of $ \W{s}(M) $ as
	\begin{equation*}
		\seminorm{f}_{\W{s}(M)} = \left(
		\int_{M \times M} \frac{\abs{f(x) - f(y)}^q}{\abs{x - y}^{n + sq}} \,\d x \d y\right)^{\oneq},
	\end{equation*}
	for $ f \in \W{s}(M) $ with $ 0 < s < 1 $, $ 1 < q < \infty $. 
	
	For an interval $ I \subset \bbr $, $ 0 < s < 1 $, $ 1 < q < \infty $, we recall the Banach space-valued
	spaces $ \K{s}(I; X) $, $ K \in \{W,H\} $ with norm
	\begin{equation*}
		\norm{f}_{\K{s}(I; X)} = \left( \norm{f}_{\Lq{q}(I; X)}^q + \seminorm{f}_{\K{s}(I; X)}^q \right)^{\oneq}.
	\end{equation*}
	For convenience, with $ T > 0 $ we define the corresponding space with vanishing initial trace at $ t = 0 $ as 
	\begin{equation*}
		\KO{s}(0,T; X) = \left\{f \in \K{s}(0,T; X) : \rv{f}_{t = 0} = 0 \right\} \text{ for } s > \frac{1}{q}.
	\end{equation*}
	
	Now we recall several embedding results and properties for Banach space-valued spaces that will be frequently used later.
	\begin{lemma}
		\label{lemma:timeembedding}
		Suppose $ 0 < r < s \leq 1 $ and $ 1 \leq q < \infty $. $ X $ is a Banach space and $ I = (0,T) \subset \bbr $ is a bounded interval for $ 0 < T < \infty $. Then $ \K{s}(I; X) \hookrightarrow \W{r}(I; X) $ and for some $ \delta > 0 $
		\begin{equation*}
			\seminorm{f}_{\W{r}(I; X)} \leq 
			\abs{I}^\delta \seminorm{f}_{\K{s}(I; X)} \text{ for all } f \in \K{s}(I; X).
		\end{equation*}
		In particular, we have
		\begin{equation*}
			\W{1}(I; X) \hookrightarrow \W{\theta}(I; X) \  \text{for all } 0 < \theta < 1 
		\end{equation*} 
		and
		\begin{equation*}
			\seminorm{f}_{\W{\theta}(I; X)} 
			\leq T^{1 - \theta} \norm{\pt f}_{\Lq{q}(I; X)} \text{ for all } f \in \W{1}(I; X).
		\end{equation*}
	\end{lemma}
	\begin{proof}
		The case $ K = W $ was shown in Simon \cite[Corollary 17]{Simon1990}. Then taking $ t $ satisfying $ r < t < s $, one has
		\begin{equation*}
			\seminorm{f}_{\W{r}(I; X)} \leq 
			\abs{I}^{t - r} \seminorm{f}_{\W{t}(I; X)}
			\leq C \abs{I}^{t - r} \seminorm{f}_{\K{s}(I; X)},
		\end{equation*}
		where the last inequality holds in virtue of $ \H{r}(I; X) \hookrightarrow \W{t}(I; X) $ for $ r > t > 0 $. The second assertion can be easily derived by means of the observation
		\begin{equation*}
			f(t) - f(t - h) = h \int_0^1 \pt f(t + (\tau - 1)h) \d \tau
		\end{equation*}
		and the definition of Sobolev--Slobodeckij space, so we omit it here.
	\end{proof}

	\begin{lemma}
		\label{lemma:timeembedding-continuous}
		Let $ 0 < s < 1 $, $ 1 < q < \infty $ satisfying $ sq > 1 $, $ X $ be a Banach space and $ I = (0,T) \subset \bbr $ be a bounded interval for $ 0 < T < \infty $. Then
		\begin{equation*}
			\K{s}(I; X) \hookrightarrow C(\bar{I}; X).
		\end{equation*}
		Moreover, for some $ \delta > 0 $ and all $ f \in \KO{s}(I; X) $,
		\begin{equation*}
			\norm{f}_{C(\bar{I}; X)} \leq C T^{\delta} \norm{f}_{\KO{s}(I; X)},
		\end{equation*}
		where $ C $ is independent of $ I $.
	\end{lemma}
	\begin{proof}
		By Meyries--Schnaubelt \cite[Proposition 2.10]{MS2012} with $ \mu = 2 $ there, one has the first assertion and for $ K = W $, $ 1/q< r < s $,
		\begin{equation*}
			\norm{f}_{C(\bar{I}; X)} \leq C \norm{f}_{\WO{r}(I; X)},
		\end{equation*}
		where $ C $ is independent of $ I $. 
		Then it follows from Lemma \ref{lemma:timeembedding} that 
		\begin{equation*}
			\norm{f}_{C(\bar{I}; X)} \leq C \norm{f}_{\WO{r}(I; X)}
			= C \big( \norm{f}_{\Lq{q}(I; X)} + \seminorm{f}_{\WO{r}(I; X)} \big)
			\leq 
			C \abs{I}^\delta \seminorm{f}_{\K{s}(I; X)},
		\end{equation*}
		for some $ \delta > 0 $.
		
	\end{proof}
	
	Adapting from Meyries--Schnaubelt \cite[Proposition 3.2]{MS2012} and Pr\"uss--Simonett \cite[Section 4.5.5]{PS2016}, we use the following time-space embedding lemma.
	\begin{lemma} 
		\label{lemma:time-space embedding}
		Let $ 1 < q < \infty $, $ 0 < \alpha, s < 2 $ and $ 0 < r < s $ , we have the embeddings
		\begin{equation*}
				\H{s}(0,T; \Lq{q}) \cap \Lq{q}(0,T; \K{\alpha}) \hookrightarrow \H{r}(0,T; \K{\alpha(1 - \frac{r}{s})}). 
		\end{equation*}
		In particular,
		\begin{gather*}
			\H{\onehalf}(0,T; \Lq{q}) \cap \Lq{q}(0,T; \W{1}) \hookrightarrow \H{\onequater}(0,T; \H{\onehalf}), \\
			\W{1}(0,T; \Lq{q}) \cap \Lq{q}(0,T; \W{2}) \hookrightarrow \H{\onehalf}(0,T; \W{1}). 
		\end{gather*}
		All these assertions remain true if one replaces $ W $- and $ H $- spaces by $ {_0W} $- and $ {_0H} $- spaces respectively, and the embedding constants in this case does not depend on $ T > 0 $.
	\end{lemma}

	Now we include results on multiplication and composition.
	\begin{lemma}[Multiplication]
		\label{lemma:multiplication}
		Let $ \OM \subset \bbr^d $, $ d \in \bbn_+ $, be a bounded Lipschitz domain. For $ f,g \in \K{s}(\OM) $ and $ sq > d $ with $ s > 0 $, $ 1 < q < \infty $, we have 
		\begin{equation*}
			\norm{fg}_{\K{s}(\OM)} \leq M_q \norm{f}_{\K{s}(\OM)} \norm{g}_{\K{s}(\OM)},
		\end{equation*}
		where $ M_q $ is a constant depending on $ q $.
	\end{lemma}
	\begin{proof}
		See \cite[Theorem 4.6.1/1 (5)]{RS1996} for the case $ K = H $ with $ q = q_1 = q_2 = 2 $ therein, \cite[Theorem 4.6.1/2 (18)]{RS1996} for the case $ K = W $ with $ p = q = q_1 = q_2 $ therein.
	\end{proof}
	
	\begin{lemma}[Composition of Bessel potential/Slobodeckij functions] 
		\label{lemma:composition-Slobodeckij}
		Let $ \OM \subset \bbr^d $, $ d \in \bbn_+ $, be a bounded domain with boundary of $ C^1 $ class. Let $ N \in \bbn_+ $, $ 0 < s < 1 $ and $ 1 \leq p < \infty $ with $ s > d/p $. Then for all $ f \in C^1 (\bbr^N) $ and every $ R > 0 $ there exists a constant $ C > 0 $ depending on $ R $ such that for all $ \bu \in K_p^s(\OM)^N $ with $ \norm{\bu}_{K_p^s(\OM)^N} \leq R $, it holds that $ f(\bu) \in K_p^s(\OM) $ and $ \norm{f(\bu)}_{K_p^s(\OM)} \leq C(R) $. Moreover, if $ f \in C^2(\bbr^N) $, then for all $ R > 0 $ there exists a constant $ L > 0 $ depending on $ R $  such that 
		\begin{equation*}
			\norm{f(\bu) - f(\bv)}_{K_p^s(\OM)} \leq L(R) \norm{\bu - \bv}_{K_p^s(\OM)^N}
		\end{equation*}
		for all $ \bu,\bv \in K_p^s(\OM)^N $ with $ \norm{\bu}_{K_p^s(\OM)^N}, \norm{\bv}_{K_p^s(\OM)^N} \leq R $.
%
	\end{lemma}
	\begin{proof}
		The first part follows from Runst--Sickel \cite[Theorem 5.5.1/1]{RS1996}. We note that in \cite{RS1996}, the function spaces act on the full space $ \bbr^d $. Here we just need to employ suitable extensions for $ \OM $ so that we can recover the case of full space. For the second part, let $ u,v $ be arbitrary two functions in $ K_p^s(\OM)^N $ with $ \norm{\bu}_{K_p^s(\OM)^N}, \norm{\bv}_{K_p^s(\OM)^N} \leq R $. By a simple calculation, one obtains
		\begin{equation}
			\label{Eqs:composition}
			(f(\bu) - f(\bv))(x) = \int_{0}^{1} Df(t\bu + (1 - t)\bv)(x) \,\d t \cdot (\bu - \bv)(x),
		\end{equation}
		where $ Df(\bu) := \partial_{u_j}f(\bu) $, $ j = 1,2,...,N $. Now let $ g(\bu,\bv)(x) := \int_{0}^{1} Df(t\bu + (1 - t)\bv)(x) \,\d t $, we have $ g(\bu,\bv) \in C^1(\bbr^N \times \bbr^N) $ since $ f(\bu) \in C^2(\bbr^N) $. Then the first part implies that 
		\begin{equation*}
			\norm{g(\bu,\bv)}_{K_p^s(\OM)} \leq C(R),
		\end{equation*}
		which completes the proof with \eqref{Eqs:composition} and the multiplication property Lemma \ref{lemma:multiplication} with $ s > d/p $. 
	\end{proof}
	\begin{remark}
		\label{remark:composition-Sobolev}
		We comment that for the case $ s = 1 $, the lemma above holds true as well due to \cite{RS1996}.
	\end{remark}
	
	In the following, an anisotropic trace lemma is introduced for a fractional power space.
	\begin{lemma}[Anisotropic trace on the boundary]
		\label{lemma:trace-time-regularity}
		Let $ 1 < q < \infty $ and $ \OM \subset \bbr^d $, $ d \in \bbn_+ $, be a bounded domain with $ \Gamma := \partial \OM $ of class $ C^1 $, $ T > 0 $, and
		\begin{equation*}
			X_T := \H{\onehalf}(0,T; \Lq{q}(\OM)) \cap \Lq{q}(0,T; \W{1}(\OM)).
		\end{equation*}
		Then there is a trace operator
		\begin{equation*}
			\gamma: X_T \rightarrow X_{\gamma,T} := \W{\onehalf - \onehalfq}(0,T;\Lq{q}(\Gamma)) \cap \Lq{q}(0,T; \W{1 - \oneq}(\Gamma)),
		\end{equation*}
		such that $ \gamma f = \rvm{f}_{\Gamma} $ for $ f \in X_T \cap C([0,T] \times \Bar{\OM}) $ and
		\begin{equation*}
			\norm{\gamma f}_{X_{\gamma,T}} \leq C \norm{f}_{X_T},
		\end{equation*}
		where $ C > 0 $ is independent of $ T $. Moreover, it is surjective and has a continuous right-inverse.
	\end{lemma}
	\begin{proof}
		By means of a coordinate transformation and a partition of unity of $ \OM $, one can easily reduce it to case of a half-space $ \bbr^{d - 1} \times \bbr_+ $. Then thanks to \cite[Theorem 4.5]{MS2012} with $ s = 1/2 $,  $ m = 1 $, $ \mu = 1 $ (see also \cite[Proposition 6.2.4]{PS2016} with $ m = 1 $, $ \mu = 1 $ and $ p = q $ there), one completes the proof.
	\end{proof}

	\section{Reformulation and main result}
	\subsection{System in Lagrangian coordinates} \label{sec:reformulation}
	In this section, we transform \eqref{Eqs:v_f-Eulerian}--\eqref{Eqs:osinitial-Eulerian} in deformed domain $ \Omega^t $ to the reference domain $ \OM $, whose definition are given in Section \ref{sec:model-description}. Let $ \phi / \hat{\phi} $ be any scalar function in $ \OM^t / \OM $ and $ \mathbf{w} / \hat{\mathbf{w}} $ be any vector-valued function in $ \OM^t / \OM $. Then one can easily derive the relations between derivatives in different configurations as
	\begin{gather}
		\label{ptu}
		\pt \hat{\mathbf{w}}(\bX,t) = \left( \pt + \bv(\bx,t) \cdot \nabla \right) \mathbf{w}(\bx,t), \\
		\label{grad}
		\nabla \phi = \inv{\hF} \hnab \hat{\phi}, \quad \nabla \mathbf{w} = \inv{\hF} \hnab \hat{\mathbf{w}}, \\
		\label{div}
		\Div \mathbf{w} = \tr ( \nabla \bw ) = \tr ( \inv{\hF} \hnab \hat{\mathbf{w}} ) = \invtr{\hF} : \hnab \hat{\mathbf{w}}, 
	\end{gather}
	where $ \hF $ is defined as in \eqref{Eqs:deformation gradient}. 
	
	Before the reformulation, we assume an \textit{isotropic growth}, which is the simplest nontrivial form for the growth tensor. It is taken as a multiple of the identity, namely,
	\begin{equation*}
		\hFsg = \hg \bbi, \tin \Os,
	\end{equation*}
	where $ \hg = \hg(\bX,t) $ is the metric of growth, a scalar function depending on the concentration of macrophages. Note that there are other possibilities for growth, see e.g. \cite{Goriely2017:growth,JC2012}, the isotropic one is taken for the sake of analysis. Then we have $ \hJsg = \hg^3 $ indicating the isotropic change of a volume element and \eqref{Eqs:grwoth-before} becomes
	\begin{equation}
		\label{Eqs:grwoth-after}
		\pt \hg = \frac{f_s^g}{3 \hrs} \hg, \tin \Os \times (0,T),
	\end{equation}
	with $ \hg(\bX, 0) = \hg^0 $.
	
	Since it follows from \eqref{Eqs:dt:detA-invA} that
	\begin{equation*}
		\pt \hJ = \tr \left( \inv{\hF} \pt \hF \right) \hJ
		= \tr \left( \inv{\hF} \hnab \hv \right) \hJ 
		= \Div \bv \hJ,
	\end{equation*}
	we have
	\begin{equation*}
		\hJf = \rv{\hJf}_{t = 0} = \det \bbi = 1, \tin \Of.
	\end{equation*}
	By the decomposition of $ \hFs $ and the incompressibility of the solid, we know that $ \hJse = 1 $ and
	\begin{equation*}
		\hJs = \hJsg = \hg^3, \tin \Os.
	\end{equation*}
	
	Similar to \cite{AL2021a}, the reformulated system now reads as:
	\begin{subequations}
		\label{Eqs:fullsystem-Lagrangian} 
		\begin{alignat}{3}
			\label{Eqs:fluid-Lagrangian}
				\hrf \pt \hvf - \hdiv \bbp_f = 0, \quad
				\invtr{\hFf} : \hnab \hvf & = 0 
					&& \tin \Of \times (0, T), \\
			\label{Eqs:cf-Lagrangian}
				\pt \hcf - \hDf \hdiv \big( \inv{\hFf} \invtr{\hFf} \hnab \hcf \big) & = 0
					&& \tin \Of \times (0, T), \\
			\label{Eqs:solid-Lagrangian}
				- \hdiv \bbp_s = 0, \quad
				\invtr{\hFs} : \hnab \hus - \int_0^t \frac{\gamma \beta}{\hrs} \hcs \,\d \tau & = 0 
					&& \tin \Os \times (0, T), \\
			\label{Eqs:cs-Lagrangian}
				\pt \hcs - \hDs \inv{\hJs} \hdiv \big( \hJs \inv{\hFs} \invtr{\hFs} \hnab \hcs \big) + \beta \hcs \big( 1 + \frac{\gamma}{\hrs} \hcs \big) & = 0 
					&& \tin \Os \times (0, T), \\
			\label{Eqs:css-Lagrangian}
				\pt \hcss - \beta \hcs + \frac{\gamma \beta}{\hrs} \hcs \hcss = 0, \quad
				\pt \hg - \frac{\gamma \beta}{3 \hrs} \hcs \hg & = 0 
					&& \tin \Os \times (0, T), \\
			\label{Eqs:vjump-Lagrangian}
				\jump{\hv} = 0, \quad
				\jump{\bbp} \hng & = 0 
					&& \ton \Gamma \times (0, T), \\
			\label{Eqs:cjump-Lagrangian}
				\jump{\hD \inv{\hF} \invtr{\hF} \hnab \hc} \hng = 0, \ 
				\zeta \jump{\hc} - \hDs \inv{\hFs} \invtr{\hFs} \hnab \hcs \cdot \hng & = 0
					&& \ton \Gamma \times (0, T), \\
			\label{Eqs:boundary-Lagrangian}
				\bbp_s \hngs = 0, \ 
				\hDs \inv{\hFs} \invtr{\hFs} \hnab \hcs \cdot \hngs & = 0 
					&& \ton \Gs \times (0, T), \\
			\label{Eqs:ofinitial-Lagrangian}
				\rv{\hvf}_{t = 0} = \vo_f, \quad 
				\rv{\hcf}_{t = 0} & = \co_f 
					&& \tin \Of, \\
			\label{Eqs:osinitial-Lagrangian}
				\rv{\hus}_{t = 0} = \hus^0, \quad
				\rv{\hcs}_{t = 0} = \co_s, \quad 
				\rv{\hcss}_{t = 0} = \hc_*^0, \quad 
				\rv{\hg}_{t = 0} & = \hg^0 
					&& \tin \Os,
		\end{alignat}
	\end{subequations}
	where $	\bbp_i := \hJ_i \hat{\bbt}_i \invtr{\hF}_i $, $ i \in \{f,s\} $, denotes the first Piola--Kirchhoff stress tensor associated with the Cauchy stress tensor $ \bbt_i $ defined in Section \ref{sec:model-description}. 
	
	\subsection{Compatibility condition and well-posedness}
	Before stating our main theorem, one still needs to impose suitable function spaces and compatibility conditions. Following the general setting of maximal regularity, e.g. \cite{AL2021a,AL2021b,PS2016}, where the basic space is $ \Lq{q}(\OM) $, we assume that
	\begin{gather*}
		\vo_f \in \Bq{2 - \frac{2}{q}}(\Of)^3 =: \Dqv, \quad
		\co \in \Bq{2 - \frac{2}{q}}(\OM \backslash \Gamma) =: \Dqc, \quad 
		\hc_*^0, \hg^0 \in \W{1}(\Os),
	\end{gather*}
	and $ \Dq := \Dqv \times \Dqc $. Moreover, the solution space are defined by $ \YT := \prod_{j = 0}^7 \YT^j $, where 
	\begin{gather*}
		\YT^1 := \W{1}(0,T; \Lq{q}(\Of)^3) \cap \Lq{q}(0,T; \W{2}(\Of)^3), \\
		\YT^2 := \H{\onehalf}(0,T; \W{1}(\Os)^3) \cap \Lq{q}(0,T; \W{2}(\Os)^3), \\
		\YT^3 := \left\{
			\begin{aligned}
				& \pi \in \Lq{q}(0,T; \W{1}(\Of)): \\
				& \qquad \qquad 
				\rv{\pi}_{\Gamma} \in \W{\onehalf - \onehalfq}(0,T; \Lq{q}(\Gamma)) \cap \Lq{q}(0,T; \W{1 -\oneq}(\Gamma))
			\end{aligned}
		\right\}, \\
		\YT^4 := \Lq{q}(0,T; \W{1}(\Os)) \cap \H{\onehalf}(0,T; \Lq{q}(\Os)),
			\\
		\YT^5 := \W{1}(0,T; \Lq{q}(\OM)) \cap \Lq{q}(0,T; \W{2}(\OM \backslash \Gamma)), \\
		\YT^6 := \W{1}(0,T; \W{1}(\Os)), \quad
		\YT^7 := \W{1}(0,T; \W{1}(\Os)).
	\end{gather*}

	Analogous to \cite{AL2021a}, the compatibility conditions for $ \hvf^0 $ and $ \hc^0 $ read as
	\begin{equation}
		\label{Eqs:compatibility}
		\begin{gathered}
			\Div \vo_f = 0, \quad
			\rv{\vo_f}_\Gamma = 0, 
			\\
			\rv{\big( \zeta \jump{\co} - \hDs \hnab \co_s \cdot \hng \big)}_\Gamma = 0, \quad 
			\rv{\jump{\hD \hnab \co} \cdot \hng}_\Gamma = 0, \quad 
			\rv{\hDs \hnab \co_s \cdot \hngs}_{\Gs} = 0,
		\end{gathered}
	\end{equation}

	Generally speaking, one does not need to assign any initial pressure for the Stokes equation. However, in this manuscript the coupling on the interface does lead to a condition on the initial fluid pressure since the solid equation is quasi-stationary and holds at $ t = 0 $. More specifically, we assume that there exists $ \hpif^0 \in \W{1 - 3/q}(\Gamma) $ and $ (\hus^0, \hpis^0) \in \W{2 - 2/q}(\Os)^3 \times \W{1 - 2/q}(\Os) $ satisfying
	\begin{equation}
		\label{Eqs:us0-smallness}
		\norm{\hnab \hus^0}_{\W{1 - \frac{2}{q}}(\Os)} 
		+ \norm{\hpis^0}_{\W{1 - \frac{2}{q}}(\Os)} 
		\leq \kappa,
	\end{equation}
	for sufficiently small $ \kappa > 0 $, such that
	\begin{equation}
		\label{Eqs:us0-equation}
		\begin{alignedat}{3}
			- \hdiv (DW(\bbi + \hnab \hus^0)) + \hnab \hpis^0 & = 0, && \tin \Os, \\
			\hdiv \hus^0 & = 0, && \tin \Os, \\
			\big(- \hpis^0 \bbi + DW(\bbi + \hnab \hus^0)\big) \hn_{\Gamma} & = \big(- \hpif^0 \bbi + \nuf (\hnab \vo_f + \hnab^\top \vo_f)\big) \hn_{\Gamma}, && \ton \Gamma, \\
			\big(- \hpis^0 \bbi + DW(\bbi + \hnab \hus^0)\big) \hn_{\Gs} & = 0, && \ton \Gs.
		\end{alignedat}
	\end{equation}
	
	\begin{remark}
		Here, the regularity for $ \hpif^0 $ on the interface $ \Gamma $ is initiated from the matched regularity of $ \hnab \vo_f $, $ \hpis^0 $ and $ DW(\bbi + \hnab \hus^0) $ on $ \Gamma $. Moreover, it coincides with the regularity of $ \hpif $ by the trace method of interpolation (see e.g. \cite[Example 3.4.9]{PS2016}), i.e.,
		\begin{equation*}
			\W{\onehalf - \onehalfq}(0,T; \Lq{q}(\Gamma)) \cap \Lq{q}(0,T; \W{1 -\oneq}(\Gamma)) \hookrightarrow
			C([0,T]; \W{1 - \frac{3}{q}}(\Gamma)).
		\end{equation*}
	\end{remark}
	\begin{remark}
		\label{remark:smallness}
		In this paper, we need the smallness assumption of initial displacement to guarantee the estimates with respect to the deformation gradient, e.g. \eqref{Eqs:Fs}, which is a key element to derive the final contraction property of the certain operator. This is because we consider the general case of $ \rvm{\hus}_{t = 0} $ and linearize the elastic equation around the identity $ \bbi $, not the initial deformation gradient $ \bbi + \hnab \hus^0 $. Specifically, one can not control $ (\hFs - \bbi) $ by a small constant only with a short time. In particular, for the case $ \hus^0 = 0 $, one knows $ \rvm{\hFs}_{t = 0} = \bbi $ and hence the estimates later is uniform with respect to time $ T > 0 $. Moreover, for initial pressure it does also need the smallness due to the sharp regularity of pressure, see e.g. \eqref{Eqs:pressureEstimate}.
	\end{remark}

	\begin{theorem}
		\label{theorem: main}
		Let $ 5 < q < \infty $ and $ \kappa > 0 $ be a sufficiently small constant. $ \OM \subset \bbr^3 $ is the domain defined above with $ \Gamma $, $ \Gs $ hypersurfaces of class $ C^3 $. Assume that $ (\vo_f, \co) \in \Dq $ satisfying the compatibility condition \eqref{Eqs:compatibility}, $ \hpif^0 \in \W{1 - 3/q}(\Gamma) $, $ \hc_*^0, \hg^0 \in \W{1}(\Os) $ and $ (\hus^0, \hpis^0) \in \W{2 - 2/q}(\Os)^3 \times \W{1 - 2/q}(\Os) $ fulfilling \eqref{Eqs:us0-smallness} and \eqref{Eqs:us0-equation}.
		Then there is a positive $ T_0 = T_0(\vo_f, \co, \hc_*^0, \hg^0, \kappa) < \infty $ such that for $ 0 < T < T_0 $, the problem \eqref{Eqs:fullsystem-Lagrangian} admits a unique solution $ (\hvf, \hus, \hpif, \hpis, \hc, \hcss, \hg) \in \YT $. Moreoever, $ \hc, \hcss, \hg \geq 0 $ if $ \co, \hc_*^0, \hg^0 \geq 0 $. 
	\end{theorem}
	Motivated by \cite{AL2021a,PS2016}, we prove Theorem \ref{theorem: main} via the Banach fixed point theorem. To be more precise, we are going to linearize \eqref{Eqs:fullsystem-Lagrangian} in the first step, show the well-posedness of the linear system, estimate the nonlinear terms in suitable function spaces with small time and then constract a contraction mapping.
	
	\begin{remark}
		\label{remark:higherdimension}
		In fact, Theorem \ref{theorem: main} still holds true in even more general dimensional case $ n \geq 2 $ as long as  $ q $ has a adapted restriction with respect to $ n $. This is also an advantage of making use of maximal regularity theory.
	\end{remark}

	\subsection{Linearization}
	\label{sec:linearization}
	Now following the linearization procedure in \cite{AL2021a}, we linearize \eqref{Eqs:fullsystem-Lagrangian} first, equate all the lower-order terms to the right-hand side and then arrive at the equivalent system: 
	\begin{subequations}
		\label{Eqs:fullsystem-Lagrangian-linear} 
		\begin{alignat}{3}
			\label{Eqs:fluid-Lagrangian-linear}
				\hrf \pt \hvf - \hdiv \bS(\hvf, \hpif) & = \bK_f
					&& \tin \Of \times (0, T), \\
			\label{Eqs:fluidmass-Lagrangian-linear}
				\hdiv \hvf & = G_f
					&& \tin \Of \times (0, T), \\
			\label{Eqs:fluidgamma-Lagrangian-linear}
				\bS(\hvf, \hpif) \hng - (D^2 W(\bbi) \hnab \us - \hpis \bbi) \hng & = \bH_f^1
					&& \ton \Gamma \times (0, T), \\
			\label{Eqs:solid-Lagrangian-linear}
				- \hdiv (D^2 W(\bbi) \hnab \us) + \hnab \hpis & = \bK_s 
					&& \tin \Os \times (0, T), \\
			\label{Eqs:solidmass-Lagrangian-linear}
				\hdiv \hus - \int_0^t \frac{\gamma \beta}{\hrs} \hcs \,\d \tau & = G_s 
					&& \tin \Os \times (0, T), \\
			\label{Eqs:solidgamma-Lagrangian-linear}
				\hus & = \bH_s^1
					&& \ton \Gamma \times (0, T), \\
			\label{Eqs:solidgs-Lagrangian-linear}
				(D^2 W(\bbi) \hnab \us - \hpis \bbi) \hngs & = \bH^2
					&& \ton \Gs \times (0, T), \\
			\label{Eqs:cf-Lagrangian-linear}
				\pt \hcf - \hDf \hDelta \hcf & = F_f^1
					&& \tin \Of \times (0, T), \\
			\label{Eqs:cfgamma-Lagrangian-linear}
				\hDf \hnab \hcf \cdot \hng & = F_f^2
					&& \ton \Gamma \times (0,T), \\
			\label{Eqs:cs-Lagrangian-linear}
				\pt \hcs - \hDs \hDelta \hcs & = F_s^1 
					&& \tin \Os \times (0, T), \\
			\label{Eqs:csgamma-Lagrangian-linear}
				\hDs \hnab \hcs \cdot \hng & = F_s^2
					&& \ton \Gamma \times (0, T), \\
			\label{Eqs:csgs-Lagrangian-linear}
				\hDs \hnab \hcs \cdot \hngs & = F^3
					&& \ton \Gs \times (0, T), \\
			\label{Eqs:css-Lagrangian-linear}
				\pt \hcss + \beta(\frac{\gamma \hc_*^0}{\hrs} - 1) \hcs = F^4, \quad
				\pt \hg - \frac{\gamma \beta \hg^0}{3 \hrs} \hcs & = F^5 
					&& \tin \Os \times (0, T), \\
			\label{Eqs:ofinitial-Lagrangian-linear}
				\rv{\hvf}_{t = 0} = \vo_f, \quad 
				\rv{\hcf}_{t = 0} & = \co_f 
					&& \tin \Of, \\
			\label{Eqs:osinitial-Lagrangian-linear}
				\rv{\hus}_{t = 0} = \hus^0, \quad 
				\rv{\hcs}_{t = 0} = \co_s, \quad 
				\rv{\hcss}_{t = 0} = \hc_*^0, \quad 
				\rv{\hg}_{t = 0} & = \hg^0
					&& \tin \Os,
		\end{alignat}
	\end{subequations}
	where $ \bS(\hvf, \hpif) := - \hpif + \nuf (\hnab \hvf + \tran{\hnab} \hvf) $ and 
	\begin{align*}
		& \bK_f = \hdiv \tk_f, \quad 
			\bK_s = \hdiv \tk_s, \\
		& G_f = - \left( \FfOinvtran \right) : \hnab \hvf, \quad
			G_s = - ( \FsOinvtran ) : \hnab \hus \\
		& \bH_f^1 = -\tk_f \hng + \ks \hng, \quad 
			\bH_s^1 = \int_0^t \hvf(\bX, \tau) \d \tau, \quad
			\bH^2 = - \ks \hngs, \\
		& F_f^1 = \hdiv \tFf, \quad 
			F_s^1 = \hdiv \tFs - \beta \hcs \left( 1 + \frac{\gamma}{\hrs} \hcs \right) - \frac{3 \hnab \hg}{\hg} \cdot \left( \hDs \inv{\hFs} \invtr{\hFs} \hnab \hcs \right), \\
		& F_f^2 = \hDs \nabla \hcs \cdot \hng - \jump{\tF} \cdot \hng, \quad
			F_s^2 = \zeta \jump{\hc} - \tFs \cdot \hng, \\
		& F^3 = -\tFs \cdot \hngs, \quad
			F^4 = - \frac{\gamma \beta}{\hrs} \hcs (\hcss - \hc_*^0), \quad
			F^5 = - \frac{\gamma \beta}{3 \hrs} \hcs \left( \hg - \hg^0 \right), 
	\end{align*}
	with
	\begin{align*}
		& \kf = - \hpif ( \FfOinv ) 
		+ \nuf \big( \inv{\hFf} \hnab \hvf + \tran{\hnab} \hvf \invtr{\hFf} \big) ( \FfOinvtran ) \\
		& \qquad \quad + \nuf \big( (\FfOinv) \hnab \hvf + \tran{\hnab} \hvf  (\FfOinvtran) \big), \\
		& \ks = - \hg^3 \hpis ( \FsOinv ) - (\hg^3 - (\hg^0)^3) \hpis \bbi - ((\hg^0)^3 - 1) \hpis \bbi \\
			& \qquad \quad + DW(\hFs) \big((\hg^0)^2 - \hg^2\big)
			+ + DW(\hFs) \big(1 - (\hg^0)^2\big)
			+ \hg^2\big(DW(\hFs) - DW(\hFs / \hg)\big) \\
		& \qquad \quad + \int_{0}^1 D^3 W\big((1-s)\bbi + s\hFs\big)(1-s) \d s (\hFs - \bbi) (\hFs - \bbi), \\
		& \tF = \hD \big( \inv{\hF} \invtr{\hF} - \bbi \big) \hnab \hc.
	\end{align*}
	\begin{remark}[Discussions on the linearization]\
		\label{remark:linearization-elastic}
		\begin{enumerate}
			\item The linearization can be derived as follows. Let $ h(s) := DW((1-s)\bbi + s\bF) $. Then $ h(0) = DW(\bbi),  h(1) = DW(\bF) $.
			Since
			\begin{equation*}
				h(1) = h(0) + h'(0) + \int_0^1 h''(s)(1-s) \,\d s,
			\end{equation*}
			it follows from \ref{assumptions:DW(I)} that
			\begin{equation*}
				DW(\bF) = D^2 W(\bbi) (\bF - \bbi) + \bR(\bF),
			\end{equation*}
			where
			\begin{equation*}
				\bR(\bF) := \int_0^1 D^3 W((1-s)\bbi + s\bF)(1-s) \,\d s (\bF - \bbi) (\bF - \bbi).
			\end{equation*}
			\item The linearization is similar to the one in \cite{AL2021a}, but with several modifications, one of which is deduced above. It is possible to have other kinds of linearizations but we remark that in the present paper a divergence structure of $ \hdiv \tk_s $ plays an essential role when we prove the linear theory and estimate it in a particular function space, see Corollary \ref{coro:f=divF} and Proposition \ref{prop:Lipschitzestimate} later. Moreover, for the solid mass balance equation \eqref{Eqs:solid-mass-Eulerian} (equivalently $ \Div \vs = {f_s^g}/{\rs} $ since the density is constant), we integrate it over $ (0,t) $ as \eqref{Eqs:solidmass-Lagrangian-linear} to keep the Stokes-type structure for the elastic equation with respect to the displacement $ \hus $. 
			\item Noticing that the continuity conditions \eqref{Eqs:vjump-Lagrangian} on the interface are separated to \eqref{Eqs:fluidgamma-Lagrangian-linear} and \eqref{Eqs:solidgamma-Lagrangian-linear} formally after the linearization, we remark here that this is for the sake of analysis due to the mismatch of the regularity on $ \Gamma $. For instant, if one replaces \eqref{Eqs:fluidgamma-Lagrangian-linear} with the boundary condition $ \hvf = \pt \hus $, it has no chance to solve the fluid part since we have no first-order temporal derivative information for the solid displacement $ \hus $.
		\end{enumerate}
	\end{remark}
	\begin{remark}
		\label{remark:G-form}
		Analogously to \cite[Remark 2.3]{AL2021a}, $ G $ also possesses the form
			\begin{equation}
				\label{Eqs:G-form}
				\begin{aligned}
					G_f & = - \hdiv \big( ( \FfOinv ) \hvf \big), \quad
					G_s = - \hdiv \big( ( \FsOinv ) \hus \big)
					+ \hus \cdot \hdiv \invtr{\hFs},
				\end{aligned}
			\end{equation}
			with the help of the Piola identity $ \hdiv(\hJ \invtr{\hF}) = 0 $.
	\end{remark}
\section{Analysis of the linear systems}
	\label{sec:analysis-linear}
	In this section, we are devoted to solve the linear systems associated with \eqref{Eqs:fullsystem-Lagrangian-linear}.
	\subsection{Nonstationary Stokes equation}
	\label{sec:nonsta-Stokes}
	Let $ \OM $ be a bounded domain with a boundary $ \partial \OM $ of class $ C^{3-} $, $ T > 0 $. We consider the nonstationary Stokes equation
	\begin{equation}
		\label{Eqs:linear-nonsta}
		\begin{alignedat}{3}
			\rho \pt \bu - \Div \Smu(\bu, \pi) & = \mathbf{f}, && \tin \OM \times (0, T), \\
			\Div \bu & = g, && \tin \OM \times (0, T), \\
			\Smu(\bu, \pi) \bn & = \bh, && \ton \partial \OM \times (0, T), \\
			\rv{\bu}_{t = 0} & = \bu_0, && \tin \OM,
		\end{alignedat}
	\end{equation}
	where $ \Smu(\bu, \pi) = - \pi \bbi + \mu (\nabla \bu + \tran{\nabla} \bu) $. $ \rho, \mu > 0 $ are the constant density and viscosity. $ \bn $ denotes the unit outer normal vector on $ \partial \OM $. Then we have the following solvability and regularity result, which can be adapted directly from e.g. Abels \cite[Theorem 1.1]{Abels2010}, Bothe--Pr\"uss \cite[Theorem 4.1]{BP2007}, Pr\"uss--Simonett \cite[Theorem 7.3.1]{PS2016} by the argument of Abels--Liu \cite[Proposition A.1]{AL2021a}.
	\begin{theorem}
		\label{thm:linear-Stoke-nonstationary}
		Let $ 3 < q < \infty $, $ T_0 > 0 $. Suppose that the initial data is $ \bu_0 \in \W{2 - 2/q}(\OM)^3 $ satisfying compatibility conditions
		\begin{gather*}
			\Div \bu_0 = \rv{g}_{t = 0}, \quad 
			\rv{\cP_{\bn}(\mu (\nabla \bu_0 + \tran{\nabla} \bu_0) \bn)}_{\partial \OM} = \rv{\bh}_{t = 0},
		\end{gather*}
		where $ \cP_{\bn} := \bbi - \bn \otimes \bn $ denotes the tangential projection onto $ \partial \OM $. For given data $ (\mathbf{f}, g, \bh) $ with
		\begin{gather*}
			\mathbf{f} \in \bbf_{\mathbf{f}}(T) := \Lq{q}(0,T; \Lq{q}(\OM)^3), \\
			g \in \bbf_g(T) := \Lq{q}(0,T; \W{1}(\OM)) \cap\W{1}(0,T; \W{- 1}(\OM)), \\
			\bh \in \bbf_\bh(T) := \Lq{q}(0,T; \W{1 - \oneq}(\partial \OM)^3) \cap \W{\onehalf - \onehalfq}(0,T; \Lq{q}(\partial \OM)^3),
		\end{gather*}
		\eqref{Eqs:linear-nonsta} admits a unique solution $ (\bu, \pi) \in \bbe(T) := \bbe_\bu(T) \times \bbe_\pi(T) $ where
		\begin{gather*}
			\bbe_\bu(T) := \Lq{q}(0,T; \W{2}(\OM)^3) \cap \W{1}(0,T; \Lq{q}(\OM)^3), \\
			\bbe_\pi(T) := \left\{
				\Lq{q}(0,T; \W{1}(\OM)):
				\rv{\pi}_{\partial \OM} \in \W{\onehalf - \onehalfq}(0,T; \Lq{q}(\partial \OM)) \cap \Lq{q}(0,T; \W{1 - \oneq}(\OM))
			\right\}.
		\end{gather*}
		Moreover, there is a constant $ C > 0 $ independent of $ \mathbf{f}, g, \bh, \bu_0, T_0 $, such that for $ 0 < T \leq T_0 $
		\begin{equation*}
			\norm{(\bu, \pi)}_{\bbe(T)} \leq C \left(
				\norm{\mathbf{f}}_{\bbf_{\mathbf{f}}(T)}
				+ \norm{g}_{\bbf_g(T)}
				+ \norm{\bh}_{\bbf_\bh(T)}
				+ \norm{\bu_0}_{\W{2 - 2/q}(\OM)}
			\right).
		\end{equation*}
	\end{theorem}
	\begin{remark}
		In our case, there will be a term of the form $ (D^2 W(\bbi) \nabla \bv - p \bbi) \bn $ in the third equation of \eqref{Eqs:linear-nonsta} with certain regularity. It is not a problem since given $ (\bv,p) $ such that $ (D^2 W(\bbi) \nabla \bv - p \bbi) \bn $ is endowed with the same regularity of $ \bh $, one can solve the original equation with $ \bh = (D^2 W(\bbi) \nabla \bv - p \bbi) \bn $ and $ (\mathbf{f}, g, \bu_0) = 0 $ by Theorem \ref{thm:linear-Stoke-nonstationary} and add the solution above to recover the case.
	\end{remark}
	\subsection{Quasi-stationary Stokes equation with mixed boundary conditions}
	\label{sec:sta-Stokes-mixed}
	Let $ \OM $ be a bounded domain with a boundary $ \partial \OM $ of class $ C^{3-} $, $ \partial \OM = \Gamma_1 \cup \Gamma_2 $ consisting of two closed, disjoint, nonempty components. Consider the generalized stationary Stokes-type equation
	\begin{equation}
		\label{Eqs:linear-sta}
		\begin{alignedat}{3}
			- \Div (D^2 W(\bbi) \nabla \bu) + \nabla \pi & = \mathbf{f}, && \tin \OM, \\
			\Div \bu & = g, && \tin \OM, \\
			\bu & = \bh^1, && \ton \Gamma_1, \\
			(D^2 W(\bbi) \nabla \bu - \pi\bbi) \bn & = \bh^2, && \ton \Gamma_2,
		\end{alignedat}
	\end{equation}
	where $ \bn $ denotes the unit outer normal vector on $ \partial \OM $. $ W : \bbr^{3 \times 3} \rightarrow \bbr_+ $ is the elastic energy density such that \ref{assumption:W-energy-density} holds. Before going to the quasi-stationary case, we first investigate the weak solution and strong solution in $ L^q $-class of the stationary problem \eqref{Eqs:linear-sta}.
	\begin{theorem}
		\label{thm:linear-sta}
		Let $ 1 < q < \infty $ and $ s \in \{0, -1\} $. Given $ \mathbf{f} \in W_{q, \Gamma_1}^{s}(\OM)^3 $, $ g \in \W{1 + s}(\OM) $, $ \bh^1 \in \W{2 + s - 1/q}(\Gamma_1)^3 $ and $ \bh^2 \in \W{1 + s - 1/q}(\Gamma_2)^3 $. Then problem \eqref{Eqs:linear-sta} admits a unique solution $ (\bu, \pi) \in \W{2 + s}(\OM)^3 \times \W{1 + s}(\OM) $. Moreover, there is a constant $ C > 0 $ such that
		\begin{equation*}
			\norm{\bu}_{\W{2 + s}(\OM)^3} + \norm{\pi}_{\W{1 + s}(\OM)} 
			\leq C \Big(
				\norm{g}_{\W{1 + s}(\OM)} 
					+ \norm{\bh^1}_{\W{2 + s - \frac{1}{q}}(\Gamma_1)^3}
					+ \NORM{\cF}_{s}
			\Big).
		\end{equation*}
		where $ \NORM{\cF}_{s} := \norm{\mathbf{f}}_{\Lq{q}(\OM)^3} + \norm{\bh^2}_{\W{1 - \oneq}(\OM)^3} $ if $ s = 0 $ and when $ s = -1 $,
		\begin{equation*}
			\NORM{\cF}_{s}
				:= \sup_{\norm{\bw}_{W_{q',\Gamma_1}^{1}(\OM)^3} = 1}
				\Big(
					\inner{\mathbf{f}}{\bw}_{W_{q,\Gamma_1}^{-1}(\OM)^3 \times W_{q',\Gamma_1}^{1}(\OM)^3} + \inner{\bh^2}{\rv{\bw}_{\Gamma_2}}_{\W{- \oneq}(\Gamma_2)^3 \times W_{q'}^{1 - \frac{1}{q'}}(\Gamma_2)^3}
				\Big).
		\end{equation*}
	\end{theorem}
	\begin{proof}
		First let $ s = 0 $, we reduce the system \eqref{Eqs:linear-sta} to the case $ (g, \bh^1, \bh^2) = 0 $. To this end, take a cutoff funtion $ \psi \in C_0^\infty((0,T)) $ such that 
		\begin{equation*}
			\int_{T/4}^{3T/4} \psi(t) \,\d t = 1, \quad \text{in } [T/4, 3T/4].
		\end{equation*}
		Then
		\begin{gather*}
			\psi(t) g \in \Lq{p}(0,T; \W{1}(\OM)) \cap W_p^{1}(0,T; W_{q, \Gamma_2}^{-1}(\OM)), \\
			\psi(t) \bh^j \in \Lq{p}(0,T; \W{3 - j - \oneq}(\Gamma_j)^3) \cap W_p^{\frac{1}{j} - \onehalfq}(0,T; \Lq{q}(\Gamma_j)^3), \
			j = 1,2.
		\end{gather*}
		In view of Remark \ref{remark:epllipticity} and the maximal $ L^q $-regularity result for the generalized Stokes problems (e.g., \cite[Theorem 4.1]{BP2007}, Pr\"uss--Simonett \cite[Theorem 7.3.1]{PS2016}), we solve the system
		\begin{equation*}
			\begin{alignedat}{3}
				\pt \bu - \Div (D^2 W(\bbi) \nabla \bu) + \nabla \pi & = 0, && \tin \OM \times (0,T), \\
				\Div \bu & = \psi(t)g, && \tin \OM \times (0,T), \\
				\bu & = \psi(t)\bh^1, && \ton \Gamma_1 \times (0,T), \\
				(D^2 W(\bbi) \nabla \bu - \pi\bbi) \bn & = \psi(t)\bh^2, && \ton \Gamma_2 \times (0,T), \\
				\rv{\bu}_{t = 0} & = 0, && \tin \OM,
			\end{alignedat}
		\end{equation*}
		with $ 3 < p < \infty $, $ 1 < q < \infty $ to get a pair of solution $ (\tilde{\bu}, \tpi) $ fulfilling 
		\begin{equation*}
			\tilde{\bu} \in W_p^{1}(0,T; \Lq{q}(\OM)^3) \cap \Lq{p}(0,T; \W{2}(\OM)^3), \quad 
			\tpi \in \Lq{p}(0,T; \W{1}(\OM)).
		\end{equation*}
		Then one infers
		\begin{equation*}
			(\bar{\bu}, \barpi): = \int_{T/4}^{3T/4} (\tilde{\bu}, \tpi)(t) \,\d t \in \W{2}(\OM)^3 \times \W{1}(\OM),
		\end{equation*}
		and
		\begin{equation*}
			\Div \bar{\bu} = g, \tin \OM, \quad 
			\rv{\bar{\bu}}_{\Gamma_1} = \bh^1, \ton \Gamma_1, \quad
			\rv{(D^2 W(\bbi) \nabla \bar{\bu} - \barpi\bbi) \bn}_{\Gamma_2} = \bh^2, \ton \Gamma_2.
		\end{equation*}
		Subtracting the solution to \eqref{Eqs:linear-sta} with $ (\bar{\bu}, \barpi) $, we are in the position to solve \eqref{Eqs:linear-sta} with $ (g, \bh^1, \bh^2) = 0 $, which can be referred to Theorem \ref{thm:sta-stokes-lambda} with $ \lambda = 0 $. Note that the case $ \lambda = 0 $ is applicable due to Remark \ref{remark:lambda=0}.
		
		Now we consider $ s = -1 $, namely the weak solution. In this case we only reduce $ (g, \bh^1) $ to zero since the Neumann boundary trace need to make sense on $ \Gamma_2 $ correctly. Concerning the Stokes equation with Dirichlet boundary condition
		\begin{equation*}
			\begin{alignedat}{3}
				- \Delta \bu + \nabla \pi & = 0, && \tin \OM, \\
				\Div \bu & = g, && \tin \OM, \\
				\bu & = \bh^1, && \ton \Gamma_1, \\
				\bu & = \mathbf{c}, && \ton \Gamma_2,
			\end{alignedat}
		\end{equation*}
		where $ \mathbf{c} > 0 $ is a constant such that
		\begin{equation*}
			\int_{\OM} g \,\d x 
			= \int_{\Gamma_1} \bh^1 \cdot \bn \,\d \mathcal{H}^2
				+ \int_{\Gamma_2} \mathbf{c} \cdot \bn \,\d \mathcal{H}^2,
		\end{equation*}
		holds, where $ \mathcal{H}^d $ with $ d \in \bbn_+ $ denotes the $ d $-dimensional Hausdorff measure. It follows from the weak solution theory for stationary Stokes equation, see e.g. Galdi--Simader--Sohr \cite[Section 5, (5.12)]{GSS2005} in Sobolev spaces, Schumacher \cite[Theorem 4.3]{Schumacher2009} in weighted Bessel potential spaces, that one obtains a unique solution denoted by $ (\bar{\bu}, \barpi) $ such that
		\begin{equation*}
			(\bar{\bu}, \barpi) 
			\in \W{1}(\OM)^3 \times \Lq{q}(\OM),
		\end{equation*}
		and
		\begin{equation*}
			\Div \bar{\bu} = g, \tin \OM, \quad 
			\rv{\bar{\bu}}_{\Gamma_1} = \bh^1, \ton \Gamma_1.
		\end{equation*}
		Then one can subtract the solution of \eqref{Eqs:linear-sta} with $ (\bar{\bu}, \barpi) $ and solve \eqref{Eqs:linear-sta} with reduced data $ (g, \bh^1) = 0 $ and modified $ (f, \bh^2) $ (not to be relabeled). The idea of the proof is to introduce a $ \Lq{q} $-class of \textit{very weak solution} (see e.g. \cite{GSS2005,Schumacher2009}), so that one can derive a solution with certain regularity in $ \W{1}(\OM) $ by complex interpolation, see e.g. Schumacher \cite{Schumacher2009} for the stationary Stokes equation in fractional Bessel potential spaces. 
		
		Define the solenoidal space
		\begin{equation*}
			\Lqs{q}(\OM) := \left\{\bu \in \Lq{q}(\OM)^3: \Div \bu = 0, \rv{\bn \cdot \bu}_{\Gamma_1} = 0\right\}.
		\end{equation*}
		For $ 1 < q, q' < \infty $ satisfying $ 1/q + 1/q' = 1 $, we define a generalized Stokes-type operator with respect to \eqref{Eqs:linear-sta} as
		\begin{equation*}
			\cA_q(\bu) := \bbp_q \big( - \Div (D^2 W(\bbi) \nabla \bu) \big) \text{ for all } \bu \in \cD(\cA_q),
		\end{equation*}
		with
		\begin{equation*}
			\cD(\cA_q) = \left\{ \bu \in \W{2}(\OM)^3 \cap \Lqs{q}(\OM): \rv{\bu}_{\Gamma_1} = 0, \ \rv{\cP_{\bn}((D^2 W(\bbi) \nabla \bu) \bn)}_{\Gamma_2} = 0 \right\},
		\end{equation*}
		where $ \bbp_q $ denotes the \textit{Helmholtz--Weyl projection} onto $ \Lqs{q}(\OM) $, see e.g. \cite[Appendix A]{Abels2010} for the existence of the projection with mixed boundary conditions. $ \cP_{\bn} := \bbi - \bn \otimes \bn $ is the tangential projection onto $ \partial \Omega $. By the result of $ s = 0 $ we see that
		\begin{equation*}
			\cA_{q} : \cD(\cA_{q}) \rightarrow \Lqs{q}(\OM) 
		\end{equation*}
		is well-defined and bijective. Then one knows that its dual operator 
		\begin{equation*}
			\cA_{q'}^* : \Lqs{q'}(\OM)' \rightarrow \cD(\cA_{q'})' 
		\end{equation*}
		is bijective as well, which gives rise to the very weak solution. Note that $ \cA_{q} $ and $ \cA_{q'}^* $ are consistent, namely, 
		\begin{equation*}
			\begin{aligned}
				& \inner{\cA_{q'}^* \bu}{\bw}_{\cD(\cA_{q'})' \times \cD(\cA_{q'})}
				= \inner{\bu}{\cA_{q'} \bw}_{\Lqs{q}(\OM) \times \Lqs{q'}(\OM)} \\
				& \qquad = \int_\OM \nabla \bu : D^2 W(\bbi) \nabla \bw \,\d x 
				= \int_\OM D^2 W(\bbi) \nabla \bu : \nabla \bw \,\d x
				= \inner{\cA_{q} \bu}{\bw}_{\Lqs{q}(\OM) \times \Lqs{q'}(\OM)},
			\end{aligned}
		\end{equation*}
		for $ \bu \in \cD(\cA_{q}) \subseteq \Lqs{q}(\OM) $, $ \bw \in \cD(\cA_{q'}) \subseteq \Lqs{q'}(\OM) $, where $ (D^2 W(\bbi))_{ij}^{kl} = (D^2 W(\bbi))_{kl}^{ij} $, $ i,j,k,l = 1,2,3 $.
		Then by the complex interpolation of operators, e.g. \cite[Theorem 2.6]{Schumacher2009}, we record that
		\begin{equation*}
			\cA_q : \big(\Lqs{q}(\OM), \cD(\cA_{q})\big)_{\left[\onehalf\right]} \rightarrow
			\big(\Lqs{q}(\OM), \cD(\cA_{q'})'\big)_{\left[\onehalf\right]}
		\end{equation*}
		is bijective. Since $ \cA_q $ admits a bounded $ \mathcal{H}^\infty $-calculus and has bounded imaginary powers, see e.g. \cite[Theorem 1.1]{Pruess2019}, complex interpolation methods can be used to describe domains of fractional power operators. By virtue of \cite[Theorem 1.1]{Pruess2019} and \cite[Theorem 2.6]{Schumacher2009}, one obtains
		\begin{gather*}
			\big(\Lqs{q}(\OM), \cD(\cA_{q})\big)_{[\onehalf]} 
				= \cD(\cA_{q}^{1/2}) 
				= W_{\sigma,\Gamma_1}^{1,q}(\OM), \\
			\big(\Lqs{q}(\OM), \cD(\cA_{q'})'\big)_{\left[\onehalf\right]} 
				= (\Lqs{q}(\OM)', \cD(\cA_{q'}))_{\left[\onehalf\right]}' 
				= \cD(\cA_{q'}^{1/2})' 
				= W_{\sigma,\Gamma_1}^{-1,q}(\OM).
		\end{gather*}
		Consequently,
		\begin{equation*}
			\cA_q : W_{\sigma,\Gamma_1}^{1,q}(\OM) \rightarrow
			W_{\sigma,\Gamma_1}^{-1,q}(\OM)
		\end{equation*}
		is bijective, which implies there exists a unique solution $ \bu \in W_{\sigma,\Gamma_1}^{1,q}(\OM) $ such that
		\begin{equation*}
			\inner{\cA_{q} \bu}{\bw} = \inner{\cF}{\bw} \text{ for all } \bw \in W_{\sigma,\Gamma_1}^{1,q'}(\OM),
		\end{equation*}
		with $ \cF \in W_{\sigma,\Gamma_1}^{-1,q}(\OM) $ defined by
		\begin{equation*}
			\inner{\cF}{\bw} := \inner{\mathbf{f}}{\bw}_{W_{q,\Gamma_1}^{-1}(\OM) \times W_{q',\Gamma_1}^{1}(\OM)} + \inner{\bh^2}{\rv{\bw}_{\Gamma_2}}_{\W{- \oneq}(\Gamma_2) \times W_{q'}^{1 - \frac{1}{q'}}(\Gamma_2)},
		\end{equation*}
		for all $ \bw \in W_{\sigma,\Gamma_1}^{1,q'}(\OM) $. Moreover, by means of the open mapping theorem, one immediately deduce the estimate
		\begin{equation*}
			\norm{\bu}_{W_{\sigma, \Gamma_1}^{1,q}(\OM)^3} 
				\leq C 
					\NORM{\cF}_{-1},
		\end{equation*}
		in which 
		\begin{equation*}
			\NORM{\cF}_{-1}
				:= \sup_{\norm{\bw}_{W_{q',\Gamma_1}^{1}(\OM)^3} = 1}
					\Big(
						\inner{\mathbf{f}}{\bw}_{W_{q,\Gamma_1}^{-1}(\OM)^3 \times W_{q',\Gamma_1}^{1}(\OM)^3} + \inner{\bh^2}{\rv{\bw}_{\Gamma_2}}_{\W{- \oneq}(\Gamma_2)^3 \times W_{q'}^{1 - \frac{1}{q'}}(\Gamma_2)^3}
					\Big).
		\end{equation*}
		
		Up to now, one still needs to recover the pressure in the very weak sense, i.e., solving
		\begin{equation}
			\label{Eqs:Laplace-W-1}
			\int_{\OM} \pi \Delta \varphi \,\d x = \inner{F}{\varphi}, \quad \forall\, \varphi \in \cD(\Delta_{q',DN}),
		\end{equation}
		where
		\begin{gather*}
			\inner{F}{\varphi} 
				:= - \inner{\mathbf{f}}{\nabla \vp} + \int_{\OM}  D^2 W(\bbi) \nabla \bu : \nabla^2 \varphi
				+ \inner{\bh^2 \cdot \bn}{\rv{\ptial{\bn} \vp}_{\Gamma_2}}, \\
			\cD(\Delta_{q',DN}) 
				:= \left\{ \psi \in W_{q'}^2(\OM) : \rv{\ptial{\bn} \psi}_{\Gamma_1} = 0, \  \rv{\psi}_{\Gamma_2} = 0 \right\}.
		\end{gather*}
		Since $ \mathbf{f} \in W_{q, \Gamma_1}^{-1}(\OM)^3 $, $ \bu \in W_{q,\Gamma_1}^{1}(\OM)^3 $ and $ \bh^2 \in \W{- 1/q}(\Gamma_2)^3 $, it is easy to verify that functional $ F $ defined above is well-defined in $ \cD(\Delta_{q',DN})' $. For every $ u \in \Lq{q'}(\OM) $, it follows from \cite[Corollary 7.4.5]{PS2016} that there exists a unique solution $ \vp(u) \in \cD(\Delta_{q',DN}) $ satisfying $ \Delta \vp = u $. Now we define $ \pi \in \Lq{q}(\OM) $ by duality as a linear functional on $ \Lq{q'}(\OM) $ acting for every $ u $ as
		\begin{equation}
			\label{Eqs:dualfunctional}
			\inner{\pi}{u} = \inner{F}{\vp}.
		\end{equation}
		Indeed $ \pi $ is the very weak solution we are looking for, since for all $ \vp \in \cD(\Delta_{q',DN}) $, we have
		\begin{equation*}
			\inner{\pi}{\Delta \vp} = \inner{\pi}{u} = \inner{F}{\vp}.
		\end{equation*}
		The uniqueness can be showed by letting $ F = 0 $ in \eqref{Eqs:dualfunctional} so that for all $ u \in \Lq{q'}(\OM) $, $ \inner{\pi}{u} = 0 $, which implies $ \pi = 0 $ a.e. in $ \OM $. Then we have the estimate
		\begin{equation*}
			\norm{\pi}_{\Lq{q}(\OM)} 
				\leq C \norm{F}_{\cD(\Delta_{q',DN})'}
				\leq C \Big(
						\norm{\bu}_{\W{1}(\OM)^3}
						+ \NORM{\cF}_{-1}
					\Big).
		\end{equation*}
		This completes the proof.
	\end{proof}
	In Theorem \ref{thm:linear-sta}, we consider the general case of data. In fact, for applications the right-hand side terms sometimes have special structure, which is of much help to derive the estimate in a concise form.
	\begin{corollary}
		\label{coro:f=divF}
		In the case of $ s = -1 $, if there is an $ \bF \in \Lq{q}(\OM)^{3 \times 3} $ such that
		\begin{equation*}
			\mathbf{f} = \Div \bF, \text{ in } \OM, \quad
			\bh^2 = - \bF \bn, \text{ on } \Gamma_2
		\end{equation*}
		holds in the sense of distribution, i.e.,
		\begin{equation}
			\label{Eqs:f=divF}
			\inner{\mathbf{f}}{\bw}_{W_{q,\Gamma_1}^{-1}(\OM)^3 \times W_{q',\Gamma_1}^{1}(\OM)^3}
			+ \inner{\bh^2}{\rv{\bw}_{\Gamma_2}}_{\W{- \oneq}(\Gamma_2)^3 \times W_{q'}^{1 - \frac{1}{q'}}(\Gamma_2)^3}
			= \inner{\bF}{\nabla \bw},
		\end{equation} 
		for all $ \bw \in W_{\sigma,\Gamma_1}^{1,q'}(\OM) $, then the solution $ (\bu, \pi) $ in Theorem \ref{thm:linear-sta} satisfies
		\begin{equation*}
			\norm{\bu}_{\W{1}(\OM)^3} + \norm{\pi}_{\Lq{q}(\OM)} 
			\leq C \Big(
				\norm{g}_{\Lq{q}(\OM)} 
				+ \norm{\bh^1}_{\W{1 - \frac{1}{q}}(\Gamma_1)^3}
				+ \norm{\bF}_{\Lq{q}(\OM)^{3 \times 3}}
			\Big).
		\end{equation*}
	\end{corollary}
	\begin{proof}
		On account of \eqref{Eqs:f=divF} and the definition of $ \NORM{\cF}_{-1} $ above with respect to $ (\mathbf{f}, \bh^2) $, one has
		\begin{equation*}
			\NORM{\cF}_{-1} = \norm{\bF}_{\Lq{q}(\OM)^{3 \times 3}},
		\end{equation*}
		which gives birth to the desired estimate.
	\end{proof}
	\begin{remark}
		In fact, Theorem \ref{thm:linear-sta} can be generalize to $ s \in (-2,0) $ by employing complex interpolation. Namely, since $ \cA_q $ admits bounded imaginary powers, we have the domains of any fractional powers by complex interpolation
		\begin{equation*}
			\cD(\cA_{q}^{\theta}) 
			= \big(\Lqs{q}(\OM), \cD(\cA_{q})\big)_{[\theta]}, \quad 0 < \theta < 1.
		\end{equation*}
		More details can be found in e.g. \cite{Abels2010,Pruess2019,Schumacher2009}.
	\end{remark}

	Combining with the temporal regularities, Theorem \ref{thm:linear-sta} and Corollary \ref{coro:f=divF}, one arrives at the following theorem.
	\begin{theorem}
		\label{thm:linear-Stoke-stationary}
		Let $ 1 < q < \infty $ and $ T_0 > 0 $. Given $ (\mathbf{f}, g, \bh^1, \bh^2) $ such that 
		\begin{gather*}
			\mathbf{f} \in \bbf_{\mathbf{f}}(T) := \Lq{q}(0,T; \Lq{q}(\OM)^3) \cap \H{\onehalf}(0,T; W_{q, \Gamma_1}^{- 1}(\OM)^3), \\
			g \in \bbf_g(T) := \Lq{q}(0,T; \W{1}(\OM)) \cap \H{\onehalf}(0,T; \Lq{q}(\OM)),
				\\
			\bh^1 \in \bbf_{\bh^1}(T) := \Lq{q}(0,T; \W{2 - \oneq}(\Gamma_1)^3) \cap \H{\onehalf}(0,T; \W{1 - \oneq}(\Gamma_1)^3), \\
			\bh^2 \in \bbf_{\bh^2}(T) := \Lq{q}(0,T; \W{1 - \oneq}(\Gamma_2)^3) \cap \H{\onehalf}(0,T; \W{- \oneq}(\Gamma_2)^3).
		\end{gather*}
		Then \eqref{Eqs:linear-sta} admits a unique solution $ (\bu, \pi) \in \bbe(T) := \bbe_\bu(T) \times \bbe_\pi(T) $ where
		\begin{gather*}
			\bbe_\bu(T) := \Lq{q}(0,T; \W{2}(\OM)^3) \cap \H{\onehalf}(0,T; \W{1}(\OM)^3), \\
			\bbe_\pi(T) := \Lq{q}(0,T; \W{1}(\OM)) \cap \H{\onehalf}(0,T; \Lq{q}(\OM))
			.
		\end{gather*}
		If additionally, there is an $ \bF \in \Lq{q}(0,T; \W{1}(\OM)^{3 \times 3}) \cap \H{1/2}(0,T; \Lq{q}(\OM)^{3 \times 3}) $ such that
		\begin{equation*}
			\mathbf{f} = \Div \bF, \text{ in } \OM, \quad
			\bh^2 = - \bF \bn, \text{ on } \Gamma_2
		\end{equation*}
		holds in the sense of distribution. Then there is a constant $ C > 0 $ independent of $ \mathbf{f}, g, \bh^1, \bh^2, T $, such that for $ 0 < T < \infty $
		\begin{equation*}
			\norm{(\bu, \pi)}_{\bbe(T)} \leq C \Big(
				\norm{\bF}_{\Lq{q}(0,T; \W{1}(\OM)^{3 \times 3}) \cap \H{\onehalf}(0,T; \Lq{q}(\OM)^{3 \times 3})}
				+ \norm{g}_{\bbf_g(T)}
				+ \norm{\bh^1}_{\bbf_{\bh^1}(T)}
			\Big).
		\end{equation*}
	\end{theorem}
	Now given $ \gamma > 0 $ and $ c \in \Lq{q}(0,T; \W{2}(\OM)) \cap \W{1}(0,T; \Lq{q}(\OM)) $, one has the solvability of the system
	\begin{equation}
		\label{Eqs:linear-sta-c}
		\begin{alignedat}{3}
			- \Div (D^2 W(\bbi) \nabla \bu) + \nabla \pi & = \mathbf{f}, && \tin \OM, \\
			\Div \bu - \gamma \int_0^t c \,\d \tau & = g, && \tin \OM, \\
			\bu & = \bh^1, && \ton \Gamma_1, \\
			(D^2 W(\bbi) \nabla \bu - \pi\bbi) \bn & = \bh^2, && \ton \Gamma_2.
		\end{alignedat}
	\end{equation}
	\begin{corollary}
		Let $ \gamma > 0 $. Given $ c \in \Lq{q}(0,T; \W{2}(\OM)) \cap \W{1}(0,T; \Lq{q}(\OM)) $. Then under the assumptions of Theorem \ref{thm:linear-Stoke-stationary}, there is a unique solution $ (u, \pi) $ of \eqref{Eqs:linear-sta-c} satisfying
		\begin{gather*}
			(u, \pi) \in \bbe(T),
		\end{gather*}
	\end{corollary}
	\begin{proof}
		Similar to \cite[Corollary 4.1]{AL2021b}, the only point we need to check is $ \gamma \int_0^t c \d \tau \in \bbf_g(T) $, which is not hard to verify thanks to the regularity of $ c $. Then solving \eqref{Eqs:linear-sta} with $ (\mathbf{f}, \bh^1, \bh^2) = 0 $ and $ g $ substituted by $ \gamma \int_0^t c \d \tau \in \bbf_g(T) $, adding the resulted solution and that of \eqref{Eqs:linear-sta}, one completes the proof.
	\end{proof}
	\begin{remark}
		In view of Theorem \ref{thm:linear-Stoke-stationary} and Lemma \ref{lemma:trace-time-regularity}, we know that 
		\begin{equation*}
			\rv{\big(D^2 W(\bbi) \nabla \bu - \pi \bbi\big) \bn}_{\Gamma_1} \in \Lq{q}(0,T; \W{1 - \oneq}(\Gamma_1)^3) \cap \W{\onehalf - \onehalfq}(0,T; \Lq{q}(\Gamma_1)^3)
		\end{equation*}
		makes sense, which contributes to the nonlinear estimate for the fluid part.
	\end{remark}
	\subsection{Heat equations with Neumann boundary condition} 
	\label{sec:heat-neumann}
	Let $ T > 0 $, $ \OM \subset \bbr^3 $ be a bounded domain with $ \partial \OM $ of class $ C^{3-} $. $ \nu $ denotes the unit outer normal vectors on $ \partial \OM $. Consider the problem
	\begin{equation}\label{Eqs:linear-heat}
		\begin{alignedat}{3}
			\pt u - D \Delta u & = f, && \tin \OM \times (0, T), \\
			D \nabla u \cdot \nu & = g, && \ton \partial \OM \times (0, T), \\
			\rv{u}_{t = 0} & = u_0, && \tin \OM,
		\end{alignedat}
	\end{equation}
	where $ D > 0 $ is a constant. $ u : \Bar{\OM} \times (0,T) \rightarrow \bbr $ stands for the system unknown, for example, the temperature, the concentration, etc. By e.g. \cite[Proposition A.2]{AL2021a}, the existence and uniqueness result reads as
	\begin{theorem}
		\label{parabolic: theorem}
		Let $ 3 < q < \infty $ and $ T_0 > 0 $. Assume that $ u_0 \in \W{2 - 2/q}(\OM) $ with the compatibility condition $ \rvm{D \nabla u_0}_{\partial \OM} = \rvm{g}_{t = 0} $ holds. Given known functions $ (f, g) $ with regularity 
		\begin{gather*}
			f \in \bbf_{f}(T) := \Lq{q}(0,T; \Lq{q}(\OM)), \\
			g \in \bbf_g(T) := \W{\onehalf - \onehalfq}(0,T; \Lq{q}(\partial \OM)) \cap \Lq{q}(0,T; \W{1 - \oneq}(\partial \OM)).
		\end{gather*}
		Then the parabolic equation \eqref{Eqs:linear-heat} admits a unique strong solution $ u \in \bbe(T) $ where 
		\begin{equation*}
			\bbe(T) := \Lq{q}(0,T; \W{2}(\OM)) \cap \W{1}(0,T; \Lq{q}(\OM)).
		\end{equation*}
		Moreover, there is a constant $ C > 0 $ independent of $ f, g, u_0, T_0 $, such that for $ 0 < T < T_0 $
		\begin{equation*}
			\norm{u}_{\bbe(T)} \leq C \left(
			\norm{f}_{\bbf_{f}(T)}
			+ \norm{g}_{\bbf_g(T)}
			+ \norm{u_0}_{\W{2 - 2/q}(\OM)}
		\right).
		\end{equation*}
	\end{theorem}

	\subsection{Ordinary differential equations for foam cells and growth}
	Let $ \OM $ be the domain defined in Section \ref{sec:heat-neumann}. Given function $ f \in \bbf(T) := \Lq{q}(0,T; \W{1}(\OM)) $, a constant $ \gamma > 0 $ and a function $ u_0 \in \W{1}(\OM) $, $ u_0 \geq 0 $.
	Then by ordinary differential equation theory, 
	\begin{equation}
		\label{Eqs:growth-linear}
		\begin{aligned}
			\pt u - \gamma w = f, \quad & \text{in}\  \OM \times (0, T), \\
			\rv{u}_{t = 0} = u_0, \quad & \text{in}\  \OM.
		\end{aligned}
	\end{equation}
	admits a unique solution
	\begin{equation*}
		u \in \bbe(T) := \W{1}(0,T; \W{1}(\OM)),
	\end{equation*}
	provided $ w \in \Lq{q}(0,T; \W{2}(\OM)) \cap \W{1}(0,T; \Lq{q}(\OM)) $. Moreover, for every $ T_0 > 0 $, there exists a constant $ C > 0 $ independent of $ f, u_0, T_0 $, such that for $ 0 < T < T_0 $
	\begin{equation*}
		\norm{u}_{\bbe(T)}
		\leq C \left(\norm{f}_{\bbf(T)} + \norm{w}_{\Lq{q}(0,T; \W{2}(\OM)) \cap \W{1}(0,T; \Lq{q}(\OM))} + \norm{u_0}_{\W{1}(\OM)}\right).
	\end{equation*}
	
\section{Nonlinear well-posdeness}
	\label{sec:nonlinear-wpd}
	We denote by $ \delta $ a universal positive function
	\begin{equation}
		\label{Eqs:deltaT}
		\delta : \bbr_+ \rightarrow \bbr_+ \text{ such that } \delta(t) \rightarrow 0_+, \text{ as } t \rightarrow 0_+.
	\end{equation}
	The most common example in the present paper is $ \delta(t) = t^\theta $ for different $ \theta > 0 $ from line to line.
	
	\subsection{Auxiliary lemmas}
	\label{sec:useful-lemmas}
	In this section, we give some lemmas which we shall use later on.
	\begin{lemma}
		\label{lemma:Bessel-multiplication}
		Let $ f,g \in \H{\onehalf}(\bbr; \Lq{q}(\OM)) \cap \Lq{\infty}(\bbr; \Lq{\infty}(\OM)) \cap W^{\alpha}_{2q}(\bbr; L^{2q}(\Omega)) $ with $ q > 1 $ and $ 1/4 < \alpha < 1/2 $, then $ fg \in \H{\onehalf}(\bbr; \Lq{q}(\OM)) $ and
		\begin{align*}
			\norm{fg}_{\H{\onehalf}(\bbr; \Lq{q}(\Omega))} 
			& \leq C\Big(
				\norm{f}_{\H{\onehalf}(\bbr; \Lq{q}(\Omega))} \norm{g}_{\Lq{\infty}(\bbr; \Lq{\infty}(\Omega))} \\
				& \quad + \norm{g}_{\H{\onehalf}(\bbr; \Lq{q}(\Omega))} \norm{f}_{\Lq{\infty}(\bbr; \Lq{\infty}(\Omega))} 
				+ \norm{f}_{W^{\alpha}_{2q}(\bbr; L^{2q}(\Omega))} \norm{g}_{W^{\alpha}_{2q}(\bbr; L^{2q}(\Omega))}
			\Big).
		\end{align*}
		Moreover, if additionally $ \rvm{f}_{t = 0} = \rvm{g}_{t = 0} = 0 $, the assertion is true as well for all $ \bbr $ substituted by an interval $ (0,T) $ and the constant $ C > 0 $ in the estimate is independent of $ T > 0 $.
	\end{lemma}
	\begin{proof}
		First let us recall the equivalent definition of Bessel potential space 
		\begin{equation*}
			\normm{f}_{\H{s}(\bbr; \Lq{q}(\OM))} = \normm{f}_{\Lq{q}(\bbr; \Lq{q}(\OM))} + \normm{(- \Delta)^{\frac{s}{2}}f}_{\Lq{q}(\bbr; \Lq{q}(\OM))},
		\end{equation*}
		where the fractional Laplace operator is represented by the singular integral
		\begin{equation*}
			(- \Delta)^{\frac{s}{2}} f(t) 
			= C_s \lim_{\epsilon \rightarrow 0} \int_{\abs{h} \geq \epsilon} \frac{\Delta_h f(t)}{\abs{h}^{1 + s}} \d h,
		\end{equation*}
		with $ C_s > 0 $ a constant depending on $ s $, $ 0 < s < 2 $ and $ \Delta_h f(t) := f(t) - f(t - h) $, see e.g. \cite[Chapter V, Section 6.10]{Stein1970}.
		For $ f,g \in H^{1/2}_q(\bbr; L^q(\Omega)) \cap L^{\infty}(\bbr; L^{\infty}(\Omega)) \cap W^{\alpha}_{2q}(\bbr; L^{2q}(\Omega)) $ with $ q > 1 $ and $ 1/4 < \alpha < 1/2 $, we see that the integrand has an algebraic decay rate greater than one with respect to $ \abs{h} $, which means the singular integral is actually integrable and one can omit the ``$ \lim $'' in the following. Then we have
		\begin{align*}
			& \norm{(- \Delta)^{\onequater}(fg)(t)}_{L^q(\Omega)}
			= C \bigg\|\int_{\bbr} \frac{\Delta_h (fg)(t)}{\abs{h}^{1 + \onehalf}} \,\d h \bigg\|_{L^q(\Omega)} \\
			& = C \bigg\| 
				g(t) \int_{\bbr} \frac{\Delta_h f(t)}{\abs{h}^{1 + \onehalf}} \,\d h
				+ \int_{\bbr} \frac{f(t) \Delta_h g(t)}{\abs{h}^{1 + \onehalf}} \,\d h 
				+ \int_{\bbr} \frac{\big(\Delta_h f(t)\big) \big(\Delta_h g(t)\big)}{\abs{h}^{1 + \onehalf}} \,\d h 
			\bigg\|_{L^q(\Omega)} \\
			& \leq C \underbrace{\left(
					\norm{g(t)(- \Delta)^{\onequater}f(t)}_{L^q(\Omega)} + \norm{f(t)(- \Delta)^{\onequater}g(t)}_{L^q(\Omega)}
				\right)}_{\leq \norm{g(t)}_{L^{\infty}(\Omega))} \normm{(- \Delta)^{\onequater}f(t)}_{L^q(\Omega)} + \norm{f(t)}_{L^{\infty}(\Omega))} \normm{(- \Delta)^{\onequater}g(t)}_{L^q(\Omega)}} \\
			& \qquad \qquad + C \underbrace{\bigg\| \int_{\bbr} \frac{\big(\Delta_h f(t)\big) \big(\Delta_h g(t)\big)}{\abs{h}^{1 + \onehalf}} \,\d h \bigg\|_{L^q(\Omega)}}_{=:I(t)}.
		\end{align*}
		Dividing the region $ \bbr $ into a neighborhood of the origin and its complement, we have
		\begin{align*}
			I(t) & \leq \int_{\abs{h} \leq 1} \norm{\big(\Delta_h f(t)\big) \big(\Delta_h g(t)\big)}_{L^q(\Omega)} \frac{1}{\abs{h}^{1 + \onehalf}} \,\d h \\
			& \qquad + \int_{\abs{h} > 1} \norm{\big(\Delta_hf(t)\big) \big(\Delta_h g(t)\big)}_{L^q(\Omega)} \frac{1}{\abs{h}^{1 + \onehalf}} \,\d h =: I_1(t) + I_2(t),
		\end{align*}
		where 
		\begin{align*}
			I_1(t) & \leq C \int_{\abs{h} \leq 1} \norm{\Delta_h f(t)}_{L^{2q}(\Omega)}\norm{\Delta_h g(t)}_{L^{2q}(\Omega)} \frac{1}{\abs{h}^{1 + \onehalf}} \,\d h \\
			& = C \int_{\abs{h} \leq 1} \norm{\Delta_h f(t)}_{L^{2q}(\Omega)} \norm{\Delta_h g(t)}_{L^{2q}(\Omega)} \left(\frac{1}{\abs{h}^{1 + \frac{q}{2} + \varepsilon \frac{q}{q'}}}\right)^{\frac{1}{q}} \left(\frac{1}{\abs{h}^{1 - \varepsilon}}\right)^{\frac{1}{q'}} \,\d h \\
			& \leq C \left(\int_{\abs{h} \leq 1} \norm{\Delta_h f(t)}_{L^{2q}(\Omega)}^q \norm{\Delta_h g(t)}_{L^{2q}(\Omega)}^q \frac{\d h}{\abs{h}^{1 +  \frac{q}{2} + \varepsilon \frac{q}{q'}}}\right)^{\oneq} \underbrace{\left(\int_{\abs{h} \leq 1} \abs{h}^{-1 + \varepsilon} \,\d h \right)^{\frac{1}{q'}}}_{\leq C_\varepsilon},
		\end{align*}
		for every $ \varepsilon > 0 $ and $ 1/q + 1/q' = 1 $. By the H\"older's inequality,
		\begin{align*}
			\int_{\bbr} \abs{I_1(t)}^q \,\d t
			& \leq C_\varepsilon \int_{\bbr} \int_{\abs{h} \leq 1} \norm{\Delta_h f(t)}_{L^{2q}(\Omega)}^q \norm{\Delta_h g(t)}_{L^{2q}(\Omega)}^q \frac{\d h \d t}{\abs{h}^{1 +  \frac{q}{2} + \varepsilon \frac{q}{q'}}} \\
			& \leq C_\varepsilon \left( \int_{\bbr} \int_{\abs{h} \leq 1} \frac{\norm{\Delta_h f(t)}_{L^{2q}(\Omega)}^{2q}}{\abs{h}^{1 +  \frac{q}{2} + \varepsilon \frac{q}{q'}}} \,\d h \d t \right)^\onehalf \left( \int_{\bbr} \int_{\abs{h} \leq 1} \frac{\norm{\Delta_h g(t)}_{L^{2q}(\Omega)}^{2q}}{\abs{h}^{1 +  \frac{q}{2} + \varepsilon \frac{q}{q'}}} \,\d h \d t\right)^\onehalf \\
			& \leq C_\varepsilon \norm{f}_{W^{\onequater + \frac{\varepsilon}{2 q'}}_{2q}(\bbr; L^{2q}(\Omega))}^q \norm{g}_{W^{\onequater + \frac{\varepsilon}{2 q'}}_{2q}(\bbr; L^{2q}(\Omega))}^q.
		\end{align*}
		Moreover,
		\begin{align*}
			\int_{\bbr} \abs{I_2(t)}^q \,\d t & \leq C \int_{\bbr} \norm{f(t)}_{L^{2q}(\Omega)}^q \norm{g(t)}_{L^{2q}(\Omega)}^q \left(\int_{\abs{h} > 1} \frac{1}{\abs{h}^{1 + \onehalf}} \,\d h\right)^q \,\d t \\
			& \leq C \int_{\bbr} \norm{f(t)}_{L^{2q}(\Omega)}^q \norm{g(t)}_{L^{2q}(\Omega)}^q \,\d t \leq C \norm{f}_{L^{2q}(\bbr; L^{2q}(\Omega))}^q \norm{g}_{L^{2q}(\bbr; L^{2q}(\Omega))}^q.
		\end{align*}
		Combining the estimate
		\begin{equation*}
			\norm{fg}_{\Lq{q}(\bbr; \Lq{q}(\OM))} \leq \norm{f}_{\Lq{q}(\bbr; \Lq{q}(\OM))} \norm{g}_{\Lq{\infty}(0,T; \Lq{\infty}(\OM))}
			\leq C \norm{f}_{\H{\onehalf}(\bbr; \Lq{q}(\OM))} \norm{g}_{\Lq{\infty}(\bbr; \Lq{\infty}(\OM))},
		\end{equation*}
		we conclude that for $ 1/4 < \alpha < 1/2 $,
		\begin{align*}
			\norm{fg}_{H^{1/2}_q(\bbr; L^q(\Omega))} 
			& \leq C\Big(
			\norm{f}_{H^{1/2}_q(\bbr; L^q(\Omega))} \norm{g}_{L^{\infty}(\bbr; L^{\infty}(\Omega))} \\
			& \quad + \norm{g}_{H^{1/2}_q(\bbr; L^q(\Omega))} \norm{f}_{L^{\infty}(\bbr; L^{\infty}(\Omega))}
			+ \norm{f}_{W^{\alpha}_{2q}(\bbr; L^{2q}(\Omega))} \norm{g}_{W^{\alpha}_{2q}(\bbr; L^{2q}(\Omega))}
			\Big).
		\end{align*}
		Now denoting by $ \cE $ the extension operator for Bessel potential spaces with vanishing initial value from \cite[Lemma 2.5]{MS2012}, one arrives at
		\begin{align*}
			& \norm{fg}_{\HO{\onehalf}(0,T; L^q(\Omega))} 
			= \norm{\cE(f)\cE(g)}_{\HO{\onehalf}(0,T; \Lq{q}(\OM))}
			\leq C \norm{\cE(f)\cE(g)}_{\HO{\onehalf}(\bbr; \Lq{q}(\OM))} \\
			& \leq C \Big(
			\norm{\cE(f)}_{\H{\onehalf}(\bbr; L^q(\Omega))} \norm{\cE(g)}_{L^{\infty}(\bbr; L^{\infty}(\Omega))} + \norm{\cE(g)}_{\H{\onehalf}(\bbr; L^q(\Omega))} \norm{\cE(f)}_{L^{\infty}(\bbr; L^{\infty}(\Omega))} \\
			& \quad \qquad 
			+ \norm{\cE(f)}_{W^{\alpha}_{2q}(\bbr; L^{2q}(\Omega))} \norm{\cE(g)}_{W^{\alpha}_{2q}(\bbr; L^{2q}(\Omega))}
			\Big) \\
			& \leq C\Big(
			\norm{f}_{\HO{\onehalf}(0,T; L^q(\Omega))} \norm{g}_{L^{\infty}(0,T; L^{\infty}(\Omega))} \\
			& \quad \qquad + \norm{g}_{\HO{\onehalf}(0,T; L^q(\Omega))} \norm{f}_{L^{\infty}(0,T; L^{\infty}(\Omega))}
			+ \norm{f}_{W^{\alpha}_{2q}(0,T; L^{2q}(\Omega))} \norm{g}_{W^{\alpha}_{2q}(0,T; L^{2q}(\Omega))}
			\Big),
		\end{align*}
		where the constant $ C > 0 $ is independent of $ T $.
	\end{proof}
	\begin{lemma}
		\label{lemma:fF} 
		Let $ \OM \subset \bbr^3 $ be a bounded domain with $ C^1 $ boundary, $ R > 0 $, $ T > 0 $ and $ 5 < q < \infty $. Given
		\begin{gather*}
			f \in X :=  \H{\onehalf}(0,T; \Lq{q}(\OM)) \cap \Lq{q}(0,T; \W{1}(\OM))
		\end{gather*}
		with $ \norm{f}_{X} \leq R $. Then $ f \in \Lq{\infty}(0,T; \Lq{\infty}(\OM)) $ and there exists some $ 1/4 < s < 1/2 - 5/4q $ such that $ f \in W^{s}_{2q}(0,T; L^{2q}(\Omega)) $. In addition,
		\begin{align}
			\label{Eqs:f-Linfty}
			& \norm{f}_{\Lq{\infty}(0,T; \Lq{\infty}(\OM))} \leq C \delt, \\
			\label{Eqs:f-W-alpha}
			& \norm{f}_{W^{s}_{2q}(0,T; \Lq{2q}(\OM))} \leq C,
		\end{align}
		provided $ \rvm{f}_{t = 0} = 0 $, where $ C > 0 $ depends on $ R $. Moreover, for $ f^1, f^2, g \in X $ with $ \rvm{(f^1 - f^2)}_{t = 0} = 0 $, $ \rvm{g}_{t = 0} = 0 $ and $ \norm{(f^i, g)}_{X \times X} \leq R $, $ i \in \{1,2\} $,
		\begin{align}
			\label{Eqs:f-Linfty-difference}
				\norm{f^1 - f^2}_{\Lq{\infty}(0,T; \Lq{\infty}(\OM))} \leq C \delt \norm{f^1 - f^2}_X, \\
				\label{Eqs:f2bF-W12Lq-difference}
				\norm{(f^1 - f^2) g}_X  
				\leq C \delt \norm{f^1 - f^2}_X,
		\end{align}
		where $ C > 0 $ depends on $ R $.
	\end{lemma}
	\begin{proof}
		For $ f \in X $, by Lemma \ref{lemma:timeembedding-continuous} and \ref{lemma:time-space embedding}, we have
		\begin{equation*}
			\H{\onehalf}(0,T; \Lq{q}(\OM)) \cap \Lq{q}(0,T; \W{1}(\OM))
			\hookrightarrow \H{\frac{1}{5}}(0,T; \H{\frac{3}{5}}(\OM))
			\hookrightarrow C([0,T]; C(\Bar{\OM})),
		\end{equation*}
		for $ q > 5 $, which means $ f \in \Lq{\infty}(0,T; \Lq{\infty}(\Omega)) $. If $ \rvm{f}_{t = 0} = 0 $, the first embedding constant above is uniform with regard to $ T $ and it follows from Lemma \ref{lemma:timeembedding-continuous} that \eqref{Eqs:f-Linfty} holds true, as well as \eqref{Eqs:f-Linfty-difference}.
		By means of the time-space embedding Lemma \ref{lemma:time-space embedding} again, one has
		\begin{equation*}
			\HO{\onehalf}(0,T; L^q(\Omega)) \cap L^q(0,T; W^1_q(\Omega))
			\hookrightarrow \HO{r}(0,T; W^{1 - 2r}_q(\Omega)),
		\end{equation*}
		for $ 1/q < r < 1/2 $, where the embedding constant does not depend on $ T > 0 $. By virtue of the embeddings $ W^{s}_q(\Omega) \hookrightarrow L^{2q}(\Omega) $ for $ s - 3/q \geq - 3/2q $ and $ H^s_q(0,T; X) \hookrightarrow W^{s - 1/2q}_{2q}(0,T; X) $, one infers the conditions
		\begin{equation*}
			1 - 2r - \frac{3}{q} \geq - \frac{3}{2q},
			\quad 
			r - \frac{1}{2q} > \frac{1}{4}.
		\end{equation*}
		Combining the inequalities above together yields that $ r $ should satisfies
		\begin{equation*}
			\frac{1}{4} + \frac{1}{2q} < r \leq \frac{1}{2} - \frac{3}{4q}.
		\end{equation*}
		Since $ q > 5 $, it is easy to verify that $ \frac{1}{2} - \frac{3}{4q} > \frac{1}{4} + \frac{1}{2q} $, which means that such $ r $ does exist and $ f \in W^{r - 1/2q}_{2q}(0,T; L^{2q}(\Omega)) $. 
		
		In addition, by means of Lemma \ref{lemma:Bessel-multiplication}, one obtains for all $ 1/4 < \alpha < 1/2 $,
		\begin{align*}
			& \norm{(f^1 - f^2) g}_{\H{\onehalf}(0,T; \Lq{q}(\Omega))} \\
			& \quad \leq C \big(\underbrace{\norm{f^1 - f^2}_{\Lq{\infty}(0,T; \Lq{\infty}(\OM))} \norm{g}_X 
			+ \norm{f^1 - f^2}_{X} \norm{g}_{\Lq{\infty}(0,T; \Lq{\infty}(\OM))}}_{\leq C(R) \delt \norm{f^1 - f^2}_X \text{ thanks to \eqref{Eqs:f-Linfty} and \eqref{Eqs:f-Linfty-difference}}} \\
			& \qquad + \norm{f^1 - f^2}_{W^{\alpha}_{2q}(\bbr; L^{2q}(\Omega))} \norm{g}_{W^{\alpha}_{2q}(\bbr; L^{2q}(\Omega))} \big),
		\end{align*}
		Now choosing $ \alpha $ such that $ 1/4 < \alpha < s < 1/2 - 5/4q $, where $ s $ is given in \eqref{Eqs:f-W-alpha}, one deduces that
		\begin{equation*}
			 \norm{(f^1 - f^2) g}_{\H{\onehalf}(0,T; \Lq{q}(\Omega))} \leq C T^{s - \alpha} \norm{f^1 - f^2}_X.
		\end{equation*}
		Moreover, 
		we have 
		\begin{align*}
			& \norm{(f^1 - f^2) g}_{\Lq{q}(0,T; \W{1}(\Omega))} \\
			& \quad = \norm{(f^1 - f^2) g}_{\Lq{q}(0,T; \Lq{q}(\Omega))} 
				+ \norm{\nabla (f^1 - f^2) g}_{\Lq{q}(0,T; \Lq{q}(\Omega))}
				+ \norm{(f^1 - f^2) \nabla g}_{\Lq{q}(0,T; \Lq{q}(\Omega))} \\
			& \quad \leq \norm{f^1 - f^2}_{\Lq{\infty}(0,T; \Lq{\infty}(\OM))} \norm{g}_X
				+ \norm{f^1 - f^2}_{\Lq{q}(0,T; \W{1}(\OM))} \norm{g}_{\Lq{\infty}(0,T; \Lq{\infty}(\OM))} \\
			& \qquad + \norm{f^1 - f^2}_{\Lq{\infty}(0,T; \Lq{\infty}(\OM))} \norm{g}_{\Lq{q}(0,T; \W{1}(\OM))} \leq C(R) \delt \norm{f^1 - f^2}_X.
		\end{align*}
		 which proves \eqref{Eqs:f2bF-W12Lq-difference}.
	\end{proof}
	\begin{remark}
		If one replaces the $ \rvm{f}_{t = 0} = 0 $ condition above by $ \normm{f}_{X} \leq \kappa $ for $ \kappa > 0 $, we still have similar estimates above with $ \delt $ substituted by $ \delt + \kappa $, which can be done by same argument as in \eqref{Eqs:Fs-I} below.
	\end{remark}
	\begin{lemma}
		\label{lemma:bF}
		Let $ q > 5 $ and $ \hF $ be the deformation gradient defined by \eqref{Eqs:deformation gradient} with respect to $ \hvf \in \YT^1 $ in $ \Of $ and $ \hus \in \YT^2 $ in $ \Os $ respectively. Assume that $ \hus^0 := \rvm{\hus}_{t = 0} \in \W{2 - 2/q}(\Os)^3 $ and $ \normm{\hnab \hus^0}_{\W{1 - 2/q}(\Os)^3} \leq \kappa $ with $ \kappa > 0 $ small enough. Then for every $ R > 0 $, there are a constant $ C = C(R) > 0 $ and a finite time $ 0 < T_R < 1 $ depending on $ R $ such that for all $ 0 < T < T_R $ and $ \norm{(\hvf, \hus)}_{\YT^1 \times \YT^2} \leq R $, $ \inv{\hF} $ exists a.e. with regularities
		\begin{equation*}
			\inv{\hFf} \in \W{1}(0,T; \W{1}(\Of)^{3 \times 3}), \quad 
			\inv{\hFs} \in \H{\onehalf}(0,T; \Lq{q}(\Os)^{3 \times 3}) \cap \Lq{q}(0,T; \W{1}(\Os)^{3 \times 3}),
		\end{equation*}
		and satisfies
		\begin{align}
			\label{Eqs:Ff}
			& \norm{\inv{\hFf}}_{\Lq{\infty}(0,T; \W{1}(\Of)^{3 \times 3})} \leq C, \quad
			\norm{\inv{\hFf} - \bbi}_{\Lq{\infty}(0,T; \W{1}(\Of)^{3 \times 3})} \leq C \delt, \\
			\label{Eqs:Fs}
			& \norm{\inv{\hFs}}_{\Lq{\infty}(0,T; \Lq{\infty}(\Os)^{3 \times 3})} \leq C, \quad
			\norm{\inv{\hFs} - \bbi}_{\Lq{\infty}(0,T; \Lq{\infty}(\Os)^{3 \times 3})} \leq C (\delt + \kappa), \\
			\label{Eqs:Ffidentity}
			& \seminorm{\FfOinv}_{\WO{r}(0,T; \W{1}(\Of)^{3 \times 3})} \leq C \delt, \quad 0 < r < 1,
		\end{align}
		Moreover, for $ \hwf \in \YT^1 $ and $ \hws \in \YT^2 $ with $ \norm{(\hwf,\hws)}_{\YT^1 \times \YT^2} \leq R $ and $ \rvm{(\hwf, \hws)}_{t = 0} = \rvm{(\hvf, \hvs)}_{t = 0} $, we have
		\begin{align}
			\label{Eqs:Ffdifference}
			& \seminorm{\inv{\hFf}(\hnab\hvf) - \inv{\hFf}(\hnab\hwf)}_{\WO{r}(0,T; \W{1}(\Of)^{3 \times 3})} \leq C \delt \norm{\hvf - \hwf}_{\YT^1}, \ 0 < r < 1,\\
			\label{Eqs:Fsdifference-Linfinity}
			& \norm{\inv{\hFs}(\hnab\hvs) - \inv{\hFs}(\hnab\hws)}_{\Lq{\infty}(0,T; \Lq{\infty}(\Os)^{3 \times 3})} \leq C (\delt + \kappa) \norm{\hvs - \hws}_{\YT^2}.
		\end{align}
	\end{lemma}
	\begin{proof}
		The proof of this lemma is similar to \cite[Lemma 4.1]{AL2021a}. However, for the solid part the regularity is a bit lower due to the quasi-stationary elastic equation. For $ \hFf $, one can refer to \cite[Lemma 4.1]{AL2021a} with Lemma \ref{lemma:timeembedding-continuous}, \ref{lemma:multiplication} and \ref{lemma:composition-Slobodeckij}, while \eqref{Eqs:Ffidentity} and \eqref{Eqs:Ffdifference} follows from Lemma \ref{lemma:timeembedding}.
		For $ \hFs $, it follows from the definition of $ \YT^2 $ that
		\begin{align*}
			\hFs - \bbi = \hnab \hus & \in \H{\onehalf}(0,T; \Lq{q}(\Os)^{3 \times 3}) \cap \Lq{q}(0,T; \W{1}(\Os)^{3 \times 3}).
		\end{align*}
		Since $ \hnab \hus^0 \in \W{1 - 2/q}(\Os)^3 \hookrightarrow \Lq{\infty}(\Os)^3 $ for $ q > 5 $ and $ \normm{\hnab \hus^0}_{\W{1 - 2/q}(\Os)^3} \leq \kappa $, one obtains from Lemma \ref{lemma:fF} that
		\begin{equation} \label{Eqs:Fs-I}
			\begin{aligned}
				\sup_{0 \leq t \leq T} & \norm{\hFs - \bbi}_{\Lq{\infty}(\Os)^{3 \times 3}} \\
				& = \sup_{0 \leq t \leq T} \norm{\hFs - (\bbi + \hnab \hus^0)}_{\Lq{\infty}(\Os)^{3 \times 3}} + \norm{\hnab \hus^0}_{\Lq{\infty}(\Os)^3}
				\leq C (\delt + \kappa)
				\leq \onehalf,
			\end{aligned}
		\end{equation}
		by taking $ T_R, \kappa > 0 $ small enough such that $ \delta(T_R) + \kappa \leq 1/(2C) $. Then by the Neumann series, $ \inv{\hFs} $ does exist. Note that $ \inv{\hFs}(\hnab \hus) = \inv{(\bbi + \hnab \hus)} $ is in the class of $ C^{\infty}(\bbr^{3 \times 3} \backslash \{- \bbi\})^{3 \times 3} $ with respect to $ \hnab \hus $, it turns out from Lemma \ref{lemma:composition-Slobodeckij} and $ q > 5 $ that
		\begin{equation}
			\label{Eqs:Fs-inv}
			\inv{(\bbi + \hnab \hus^0)} \in \W{1 - \frac{2}{q}}(\Os)^{3 \times 3}, \quad 
			\inv{\hFs} \in \H{\onehalf}(0,T; \Lq{q}(\Os)^{3 \times 3}) \cap \Lq{q}(0,T; \W{1}(\Os)^{3 \times 3}).
		\end{equation}
		In addition, we have
		\begin{gather*}
			\inv{(\bbi + \hnab \hus^0)} - \bbi 
			= \int_0^1 \frac{\d}{\d \tau} \inv{(\bbi + \tau \hnab \hus^0)} \,\d \tau
			= \int_0^1 - \inv{(\bbi + \tau \hnab \hus^0)} \hnab \hus^0 \inv{(\bbi + \tau \hnab \hus^0)} \,\d \tau,
		\end{gather*}
		then
		\begin{equation*}
			\norm{\inv{(\bbi + \hnab \hus^0)} - \bbi}_{\W{1 - \frac{2}{q}}(\Os)^{3 \times 3}} \leq C \kappa,
		\end{equation*}
		provided $ \kappa > 0 $ sufficiently small, where $ C > 0 $ is finite.
		Consequently, similarly to \eqref{Eqs:Fs-I}, one can derive that
		\begin{equation*}
			\sup_{0 \leq t \leq T} \norm{\inv{\hFs} - \bbi}_{\Lq{\infty}(\Os)^{3 \times 3}}
			\leq C (\delt + \kappa),
		\end{equation*}
		which proves \eqref{Eqs:Fs} and \eqref{Eqs:Fsdifference-Linfinity}. 
	\end{proof}
	
	\begin{lemma}
		\label{lemma:DW}
		Let $ \OM \subset \bbr^3 $ be a bounded domain with $ C^1 $ boundary, $ 0 < s \leq 1 $ and $ 1 < q < \infty $ with $ sq > 3 $. Let $ W : \bbr^{3 \times 3} \rightarrow \bbr_+ $ satisfy Assumption \ref{assumptions:smoothness}. Then for $ \bF \in \K{s}(\OM)^{3 \times 3} $, $ K \in \{W, H\} $, with $ \norm{\bF}_{\K{s}(\OM)^{3 \times 3}} \leq R $, there is a positive constant $ C $ depending on $ R $ such that
		\begin{equation*}
			\norm{D^k W(\bF)}_{\K{s}(\OM)^{3^{2k}}} \leq C, \quad k \in \{0, 1, 2, 3\}.
		\end{equation*}
		Moreover, for $ \bF^1, \bF^2 \in \K{s}(\OM)^{3 \times 3} $ with $ \norm{\bF^1}_{\K{s}(\OM)^{3 \times 3}}, \norm{\bF^2}_{\K{s}(\OM)^{3 \times 3}} \leq R $, we have
		\begin{equation*}
			\norm{D^k W(\bF^1) - D^k W(\bF^2)}_{\K{s}(\OM)^{3^{2k}}} \leq C \norm{\bF^1 - \bF^2}_{\K{s}(\OM)^{3 \times 3}}, \quad k \in \{1, 2, 3\}.
		\end{equation*}
	\end{lemma}
	\begin{proof}
		One can prove it by Lemma \ref{lemma:composition-Slobodeckij} for $ 0 < s < 1 $ directly, while the case $ s = 1 $ follows by Remark \ref{remark:composition-Sobolev}.
	\end{proof}
	\begin{lemma}
		\label{lemma:gpositive}
		Let $ T > 0 $, $ R > 0 $ and $ q \in (1, \infty) $. $ \Os $ is the domain defined in Section \ref{sec:model-description}. Given $ \hg \in \W{1}(0,T; \W{1}(\Os)) $ with $ \norm{\hg}_{\W{1}(0,T; \W{1}(\Os))} \leq R $, $ \rvm{\hg(\bX, t)}_{t = 0} = \hg^0 $, there exists a time $ T_R > 0 $ such that for $ T \in (0, T_R) $, one has 
		\begin{equation*}
			\hg(\bX, t) \geq \onehalf, \ \forall \, \bX \in \Os, t \in [0,T].
		\end{equation*}
	\end{lemma}
	\begin{proof}
		By calculus,
		\begin{equation*}
			\hg(\bX, t) = \hg^0(\bX) + \int_0^t \pt \hg(\bX, \tau) \d \tau, \ \forall \, \bX \in \Os, t \in [0,T].
		\end{equation*}
		Then
		\begin{equation*}
			\norm{\hg(t) - \hg^0}_{\Lq{\infty}(\Os)} 
			\leq C \norm{\int_0^t \pt \hg(\cdot, \tau) \d \tau}_{\W{1}(\Os)}
			\leq C T^{1 - \frac{1}{q}} R \leq \onehalf,
		\end{equation*}
		where we choose $ T_R > 0 $ sufficiently small such that $ T_R^{1 - \frac{1}{q}} \leq \frac{1}{2CR} $. Hence $ \hg(\bX, t) \geq \onehalf $, for all $ \bX \in \Os, t \in [0,T] $.
	\end{proof}
	\subsection{Lipschitz estimates}
	\label{sec:lipschitz-estimates}
	Now we are in the position to derive the Lipschitz estimates of the nonlinear lower-order terms in \eqref{Eqs:fullsystem-Lagrangian-linear}. 
	To this end, let us first define the function spaces for the nonlinear terms $ \ZT := \prod_{j = 1}^{12} \ZT^j $, where
	\begin{gather*}
		\ZT^1 := \Lq{q}(0,T; \Lq{q}(\Of)^3), \\
		\ZT^2 := \Lq{q}(0,T; \Lq{q}(\Os)^3) \cap \H{\onehalf}(0,T; W_{q,\Gamma}^{-1}(\Os)^3), \\
		\ZT^3 := \left\{
					g \in \Lq{q}(0,T; \W{1}(\Of)) \cap\W{1}(0,T; \W{- 1}(\Of)): 
					\tr_\Gamma(g) \in \ZT^5
				\right\}, \\
		\ZT^4 := \Lq{q}(0,T; \W{1}(\Os)) \cap \H{\onehalf}(0,T; \Lq{q}(\Os)),
				\\
		\ZT^5 := \Lq{q}(0,T; \W{1 - \oneq}(\Gamma)^3) \cap \W{\onehalf - \onehalfq}(0,T; \Lq{q}(\Gamma)^3), \\
		\ZT^6 := \Lq{q}(0,T; \W{2 - \oneq}(\Gamma)^3) \cap \H{\onehalf}(0,T; \W{1 - \oneq}(\Gamma)^3), \\
		\ZT^7 := \Lq{q}(0,T; \W{1 - \oneq}(\Gs)^3) \cap \H{\onehalf}(0,T; \W{- \oneq}(\Gs)^3), \\
		\ZT^8 := \Lq{q}(0,T; \Lq{q}(\OM \backslash \Gamma)), \\
		\ZT^9 := \Lq{q}(0,T; \W{1 - \oneq}(\Gamma)) \cap \W{\onehalf - \onehalfq}(0,T; \Lq{q}(\Gamma)),\\
		\ZT^{10} := \Lq{q}(0,T; \W{1 - \oneq}(\Gs)) \cap \W{\onehalf - \onehalfq}(0,T; \Lq{q}(\Gs)), \\
		\ZT^{11} := \Lq{q}(0,T; \W{1}(\Os)), \quad
		\ZT^{12} := \Lq{q}(0,T; \W{1}(\Os)),
	\end{gather*}
	Let 
	\begin{equation}
		\label{Eqs:bw}
		\bw := (\hvf, \hus, \hpif, \hpis, \hc, \hcss, \hg)
	\end{equation}
	be in the space $ \YT $ given in Section \ref{sec:reformulation} and define the associated initial data as 
	\begin{equation}
		\label{Eqs:bw0}
		\bw_0 := \rv{(\hvf, \hus, \hpis, \hc, \hcss, \hg)}_{t = 0} = (\vo_f, \hus^0, \hpis^0, \co, \hc_*^0, \hg^0).
	\end{equation}
	Then we have the following Lipschitz estimates for the lower-order terms defined in \eqref{Eqs:fullsystem-Lagrangian-linear}.
	\begin{proposition}
		\label{prop:Lipschitzestimate}
		Let $ q > 5 $ and $ R > 0 $. There exist constants $ C, \kappa > 0 $ and a finite time $ T_R > 0 $ both depending on $ R $ such that for $ 0 < T < T_R $ and $ \normm{\hnab \hus^0}_{\W{1 - 2/q}(\Os)^3} + \normm{\hpis^0}_{\W{1 - 2/q}(\Os)} \leq \kappa $,
		\begin{gather*}
			\norm{\bK_f(\bw^1) - \bK_f(\bw^2)}_{\ZT^1} \leq C (\delt + \kappa) \norm{\bw^1 - \bw^2}_{\YT}, \\
			\norm{(\bK_s, \bH^2)(\bw^1) - (\bK_s, \bH^2)(\bw^2)}_{\ZT^2 \times \ZT^7} \leq C (\delt + \kappa) \norm{\bw^1 - \bw^2}_{\YT}, \\
			\norm{G_i(\bw^1) - G_i(\bw^2)}_{\ZT^3 \times \ZT^4} \leq C (\delt + \kappa) \norm{\bw^1 - \bw^2}_{\YT}, \ i \in \{f,s\}, \\
			\norm{\bH_i^1(\bw^1) - \bH_i^1(\bw^2)}_{\ZT^5 \times \ZT^6} \leq C (\delt + \kappa) \norm{\bw^1 - \bw^2}_{\YT}, \ i \in \{f,s\}, \\
			\norm{F^j(\bw^1) - F^j(\bw^2)}_{\ZT^{7 + j}} \leq C (\delt + \kappa) \norm{\bw^1 - \bw^2}_{\YT}, \ j \in \{1,2,3,4,5\},
		\end{gather*}
		for all $ \norm{\bw^1}_{\YT}, \norm{\bw^2}_{\YT} \leq R $ with $ \bw_0^1 = \bw_0^2 $.
	\end{proposition}
	\begin{proof}
		For the estimates related to the fluid (with a subscript $ f $) except $ \bH_f^1 $ and $ F^j $, $ j = 1,...,5 $, we refer to \cite[Proposition 4.2]{AL2021a}, with the help of Lemma \ref{lemma:fF} and \ref{lemma:bF}.
		
		\textbf{Estimate of $ \bK_s, \bH^2 $}. Thanks to the divergence form from the linearization, i.e. $ \bK_s = \Div \ks $ and $ \bH^2 = - \ks \hn_{\Gs} $, one can estimate $ \bK_s, \bH^2 $ together with cancellation of boundary data like Corollary \ref{coro:f=divF}. Namely,
		\begin{align*}
			& \norm{(\bK_s, \bH^2)(\bw^1) - (\bK_s, \bH^2)(\bw^2)}_{\ZT^2 \times \ZT^7} \\
			& \qquad \qquad \qquad 
				\leq \norm{\ks(\bw^1) - \ks(\bw^2)}_{\Lq{q}(0,T; \W{1}(\OM)^{3 \times 3}) \cap \H{\onehalf}(0,T; \Lq{q}(\Os)^{3 \times 3})}.
		\end{align*}
		Combining the definition of $ \ks $ in Section \ref{sec:linearization}, Lemma \ref{lemma:fF}--\ref{lemma:gpositive} and Assumption $ \ref{assumption:W-energy-density} $, one obtains for $ q > 5 $,
		\begin{equation*}
			\norm{(\bK_s, \bH^2)(\bw^1) - (\bK_s, \bH^2)(\bw^2)}_{\ZT^2 \times \ZT^7}
			\leq C (\delt + \kappa) \norm{\bw^1 - \bw^2}_{\YT}.
		\end{equation*}
		In this estimate, the smallness of the initial solid pressure is employed for the term $ ((\hg^0)^3 - 1) \hpis $ in $ \ks $, i.e.,
		\begin{equation}
			\label{Eqs:pressureEstimate}
			\norm{\hpis}_{\Lq{\infty}(0,T; \Lq{\infty}(\Os))} \leq
			\norm{\hpis - \hpis^0}_{\Lq{\infty}(0,T; \Lq{\infty}(\Os))} + \norm{\hpis^0}_{\Lq{\infty}(\Os)}
			\leq C(\delt + \kappa).
		\end{equation}
	
		\textbf{Estimate of $ G_s $}. By means of Lemma \ref{lemma:fF} and \ref{lemma:bF}, estimate in $ \Lq{q}(0,T; \W{1}(\Os)) $ is clear. By the definition of $ G_s $, 
		\begin{align*}
			G_s(\bw^1) - G_s(\bw^2)
			= (\inv{\hFs}(\bw^1) - \bbi) : (\hnab\hus^1 - \hnab\hus^2)
			+\big(\inv{\hFs}(\bw^1) - \inv{\hFs}(\bw^2)\big) : \hnab \hus^2.
		\end{align*}
		Then with Lemma \ref{lemma:fF}, \ref{lemma:bF} and the regularity \eqref{Eqs:Fs-inv}, we have for $ q > 5 $ that
		\begin{equation*}
			\norm{G_s(\bw^1) - G_s(\bw^2)}_{\H{\onehalf}(0,T; \Lq{q}(\Os))} 
			\leq C (\delt + \kappa) \norm{\bw^1 - \bw^2}_{\YT}.
		\end{equation*}
		
		\textbf{Estimate of $ \bH_{f/s}^1 $}. With $ \hvf \in \YT^1 $, one knows
		\begin{equation*}
			\int_0^t \hvf(\bX, \tau) \d \tau \in \W{1}(0,T; \W{2}(\Of)^3) \cap \W{2}(0,T; \Lq{q}(\Of)^3)
			\hookrightarrow \H{\frac{3}{2}}(0,T; \W{1}(\Of)^3).
		\end{equation*}
		It follows from the trace theorem $ \tr_{\Gamma}: \W{k}(\Of) \rightarrow \W{k - \oneq}(\Gamma) $ that
		\begin{equation*}
			\rv{\int_0^t \hvf(\bX, \tau) \d \tau}_{\Gamma} \in \H{\frac{3}{2}}(0,T; \W{1 - \oneq}(\Gamma)^3) \cap \W{1}(0,T; \W{2 - \oneq}(\Gamma)^3),
		\end{equation*}
		whose Lipschitz estimate in $ \ZT^6 $ can be controlled by $ \delta(T) \norm{\bw^1 - \bw^2}_{\YT} $ thanks to Lemma \ref{lemma:timeembedding} and \ref{lemma:trace-time-regularity}, namely,
		\begin{equation*}
		 	\norm{\bH_s^1(\bw^1) - \bH_s^1(\bw^2)}_{\ZT^6} \leq C \delt \norm{\bw^1 - \bw^2}_{\YT}.
		\end{equation*}
	 	
	 	Now let us recall $ \bH_f^1 = - \kf \hng + \ks \hng $. The first part can be addressed in the same way as in \cite[Proposition 4.2]{AL2021a}. By Lemma \ref{lemma:trace-time-regularity} the anisotropic trace theorem and $ C^3 $ interface that ensures a $ \hng $ of class $ C^2 $, the second term can be estimated by
	 	\begin{equation*}
	 		\norm{\ks(\bw^1) - \ks(\bw^2)}_{\ZT^5}
	 		\leq C \norm{\ks(\bw^1) - \ks(\bw^2)}_{\Lq{q}(0,T; \W{1}(\OM)^{3 \times 3}) \cap \H{\onehalf}(0,T; \Lq{q}(\Os)^{3 \times 3})}
	 	\end{equation*}
	 	Then together with Lemma \ref{lemma:fF}--\ref{lemma:gpositive} and Assumption $ \ref{assumption:W-energy-density} $, one gets
 		\begin{equation*}
 			\norm{\bH_f^1(\bw^1) - \bH_f^1(\bw^2)}_{\ZT^5} \leq C (\delt + \kappa) \norm{\bw^1 - \bw^2}_{\YT}.
 		\end{equation*}
 		
 		
 		\textbf{Estimate of $ F^j $}. We can estimate $ F^1_f $ and $ F^j, j = 4, 5 $ analogously as in \cite{AL2021a}. For others, since $ \YT^5 \hookrightarrow \H{1/2}(0,T; \W{1}(\Omega)) $ implies that
 		\begin{equation*}
 			\hnab \hcs \in \H{\onehalf}(0,T; \Lq{q}(\Os)^3) \cap \Lq{q}(0,T; \W{1}(\Os)^3),
 		\end{equation*}
 		one can apply Lemma \ref{lemma:fF} and \ref{lemma:bF} to derive the corresponding estimates with $ q > 5 $, combining Lemma \ref{lemma:trace-time-regularity}, the regularity of $ \tran{\hFs} $, $ \hcss $ and $ \hg $.

 		This completes the proof.
	\end{proof}
	
	\subsection{Nonlinear well-posedness}
	\label{sec:nonlinear-proof}
	For a $ \bw $ defined in \eqref{Eqs:bw}, define
	\begin{equation*}
		\sM(\bw) := \left( \bK_f, \bK_s, G_f, G_s, \bH_f^1, \bH_s^1, \bH^2, F^1, F^2, F^3, F^4, F^5 \right)^\top (\bw),
	\end{equation*}
	where the elements are given by \eqref{Eqs:fullsystem-Lagrangian-linear}. Then the following proposition holds for $ \sM(\bw) : \YT \rightarrow \ZT $, where $ \YT, \ZT $ are given in Section \ref{sec:reformulation}, \ref{sec:lipschitz-estimates} respectively.
	\begin{proposition}
		\label{prop:nonlinearestimate}
		Let $ q > 5 $ and $ R > 0 $. 		
		Let $ \bw \in \YT $ be the function as in \eqref{Eqs:bw} with the associated initial data as $ \bw_0 $ as in \eqref{Eqs:bw0}. Then there exist constants $ C, \kappa > 0 $, a finite time $ T_R > 0 $ both depending on $ R $ and $ \delt $ as in \eqref{Eqs:deltaT} such that for $ 0 < T < T_R $, $ \sM : \YT \rightarrow \ZT $ is well-defined and bounded together with the estimates for $ \norm{\bw}_{\YT} \leq R $ and $ \normm{\hnab \hus^0}_{\W{1 - 2/q}(\Os)^3} + \normm{\hpis^0}_{\W{1 - 2/q}(\Os)} \leq \kappa $ that
		\begin{equation}
			\label{Eqs:Mw}
			\norm{\sM(\bw)}_{\ZT}
			\leq C (\delt + \kappa).
		\end{equation}
		Moreover, for $ \bw^1, \bw^2 \in \YT $ with $ \bw_0^1 = \bw_0^2 $ and $ \norm{\bw^1}_{\YT} \norm{\bw^2}_{\YT} \leq R $, there exist a constant $ C > 0 $, a finite time $ T_R > 0 $ depending on $ R $ and a function $ \delt $ as in \eqref{Eqs:deltaT} such that for $ 0 < T < T_R $,
		\begin{equation}
			\label{Eqs:MLipschitz}
			\norm{\sM(\bw^1) - \sM(\bw^2)}_{\ZT} \leq C (\delt + \kappa) \norm{\bw^1 - \bw^2}_{\YT}.
		\end{equation}
	\end{proposition}
	\begin{proof}
		The second part follows directly from Proposition \ref{prop:Lipschitzestimate}. Then by setting $ \bw^2 = (0,0,0,0,0,0,1) $ in \eqref{Eqs:MLipschitz}, one derives \eqref{Eqs:Mw} immediately in view of the fact that $ \sM(0,0,0,0,0,0,1) = 0 $.
	\end{proof}
	Now recalling the definition of solution and initial spaces in Section \ref{sec:reformulation}, we rewrite \eqref{Eqs:fullsystem-Lagrangian-linear} in the abstract form
	\begin{equation}
		\label{abstract}
		\sL (\bw) = \sN(\bw, \wo) \quad \textrm{for all}\ \bw \in \YT,\ (\hvf^0, \co) \in \Dq,
	\end{equation}
	where $ \sL(\bw) $ denotes the left-hand side of \eqref{Eqs:fullsystem-Lagrangian-linear} and $ \sN(\bw, \wo) $ is the right-hand side.
	It follows from the linear theory in Section \ref{sec:analysis-linear} that $ \sL : \YT \rightarrow \ZT \times \Dq $ is an isomorphism.
	\begin{proof}[\bf Proof of Theorem \ref{theorem: main}]
		For $ (\hvf^0, \co) \in \Dq $ satisfying the compatibility conditions, we may solve $ \sL(\tilde{\bw}) = \sN(0, \wo) $ by some $ \tilde{\bw} \in \YT $.
		Then one can reduce the system to the case of trivial initial data by eliminating $ \tilde{\bw} $ and we are able to set a well-defined constant
		\begin{equation*}
			C_{\sL} := \sup_{0 \leq T \leq 1} \norm{\sL^{-1}}_{\cL(\ZT, \YT)},
		\end{equation*}
		which can be verified to stay boundeded as $ T \rightarrow 0 $ by the linear theories in Section \ref{sec:analysis-linear} and the estimate \eqref{Eqs:Mw}, as in \cite{AL2021a}. Choose $ R > 0 $ large such that $ R \geq 2 C_\sL \norm{(\vo_f, \co)}_{\Dq} $. Then 
		\begin{equation}
			\label{L0}
			\norm{\sL^{-1} \sN(0, \wo)}_{\YT} 
			\leq C_\sL \norm{(\vo_f, \co)}_{\Dq}
			\leq \frac{R}{2}.
		\end{equation}
		For $ \norm{\bw^i}_{\YT} \leq R $, $ i = 1, 2 $, we take $ T_R > 0 $ and $ \kappa > 0 $ small enough such that $ C_\sL C(R) (\delta(T_R) + \kappa) \leq 1/2 $,
		where $ C(R) $ is the constant in \eqref{Eqs:MLipschitz}. Then for $ 0 < T < T_R $, we infer from Theorem \ref{prop:nonlinearestimate} that
		\begin{equation}
			\begin{aligned}
				\label{L12}
				& \norm{\sL^{-1}\sN(\bw^1, \wo) - \sL^{-1}\sN(\bw^2, \wo)}_{\YT}  \\
				& \qquad \qquad \qquad \leq C_\sL C(R) \TD \norm{w^1 - w^2}_{\YT} 
				\leq \onehalf \norm{w^1 - w^2}_{\YT},
			\end{aligned}
		\end{equation}
		which implies the contraction property.
		From \eqref{L0} and \eqref{L12}, we have
		\begin{align*}
			& \norm{\sL^{-1}\sN(\bw, \wo)}_{\YT} \\
			& \qquad \qquad \leq \norm{\sL^{-1}\sN(0, \wo)}_{\YT} + \norm{\sL^{-1}\sN(w, \wo) - \sL^{-1}\sN(0, \wo)}_{\YT}
			\leq R.
		\end{align*}
		Define a ball in $ \YT $ as
		\begin{equation*}
			\cM_{R,T} := \left\{ 
				\bw \in \overline{B_{\YT}(0,R)}: \bw, \wo \text{ are as in } \eqref{Eqs:bw} \text{ and } \eqref{Eqs:bw0} 
			\right\},
		\end{equation*}
		a closed subset of $ \YT $. Hence, $ \sL^{-1}\sN : \cM_{R,T} \rightarrow \cM_{R,T} $ is well-defined for all $ 0 < T < T_R $ and a strict contraction. Since $ \YT $ is a Banach space, the Banach fixed-point Theorem implies the existence of a unique fixed-point of $ \sL^{-1}\sN $ in $ \cM_{R,T} $, i.e., \eqref{Eqs:fullsystem-Lagrangian-linear} admits a unique strong solution in $ \cM_{R,T} $ for small time $ 0 < T < T_R $.
		
		The uniqueness in $ \YT $, $ 0 < T < T_0 $, follows easily by repeating the continuity argument in \cite[Proof of Theorem 2.1]{AL2021a}, so we omit it here. 
		In summary, \eqref{Eqs:fullsystem-Lagrangian-linear} admits a unique solution in $ \YT $, equivalently, \eqref{Eqs:fullsystem-Lagrangian} admits a unique solution in $ \YT $.
	
		Now we are in the position to prove the positivity of cells concentrations. Since the regularity of $ \hvs = \ptial{t} \hus $ is much lower than that in \cite{AL2021a}, we can not proceed as in \cite{AL2021a} for $ \hc $. To overcome this problem, we take a smooth mollification $ \hvs^\epsilon $
		of $ \hvs $ for $ \epsilon > 0 $ such that 
		\begin{equation*}
			\int_0^t \hvs^\epsilon (\cdot, \tau) \d \tau \rightarrow \hus, \text{ in } \YT^2, \text{ as } \epsilon \rightarrow 0. 
		\end{equation*}
		Consider the problem 
		\begin{alignat*}{3}
			\pt c_f^\epsilon + \Div \left( c_f^\epsilon \vf \right) - D_f \Delta c_f^\epsilon & = 0, && \tin Q_f^T, \\ 
			\pt c_s^\epsilon + \Div \left( c_s^\epsilon \vs^\epsilon \right) - D_s \Delta c_s^\epsilon & = - f_s^r, && \tin Q_s^T,  
		\end{alignat*}
		with boundary and initial values as in Section \ref{sec:model-description}. Then with same argument in \cite[Proof of Theorem 2.1]{AL2021a}, one obtains
		\begin{equation*}
			0 \leq c^\epsilon(\bx,t) \in \W{1}(0,T; \Lq{q}(\Omega^t)^3) \cap \Lq{q}(0,T; \W{2}(\Omega^t \backslash \Gamma^t)^3), 
		\end{equation*}
		which means there is a subsequent still denoted by $ c^\epsilon $ and a function $ c $ such that
		\begin{equation*}
			c^\epsilon \rightharpoonup c \text{ weakly in } \W{1}(0,T; \Lq{q}(\Omega^t)^3) \cap \Lq{q}(0,T; \W{2}(\Omega^t \backslash \Gamma^t)^3), 
		\end{equation*}
		and 
		\begin{equation*}
			c^\epsilon \rightarrow c \text{ in }  \cD'(Q^T \backslash S^T),
		\end{equation*}
		where $ \cD'(U) $ denotes the space of distributions on $ U $, $ Q^T, S^T $ are defined in Section \ref{sec:model-description}. It is standard to verify $ c $ solves the same equation with $ \vs^\epsilon $ replaced by $ \vs $. We only give the sketch of proof with respect to $ \Div(c_s^\epsilon \vs^\epsilon) $ as an example.
		\begin{align*}
			& \int_{0}^T \int_{\Ost} c_s^\epsilon \vs^\epsilon \cdot \nabla \phi \,\d \bx \d t 
			- \int_{0}^T \int_{\Ost} c_s \vs \cdot \nabla \phi \,\d \bx \d t \\ 
			& \qquad 
			= \int_{0}^T \int_{\Ost} (c_s^\epsilon - c_s) \vs^\epsilon \cdot \nabla \phi \,\d \bx \d t 
			+ \int_{0}^T \int_{\Ost} c_s (\vs^\epsilon - \vs) \cdot \nabla \phi \,\d \bx \d t 
			\rightarrow 0,
		\end{align*}
		as $ \epsilon \rightarrow 0 $, for all $ \phi \in \cD(Q^T \backslash S^T) $, because of the regularity and convergence of $ c_s^\epsilon $ and $ \vs^\epsilon $.
		Note that
		\begin{equation*}
			0 \leq \int_{0}^T \int_{\Omega^t \backslash \Gamma^t} c^\epsilon \phi \,\d \bx \d t  
			\leq \limsup_{\epsilon \rightarrow 0} \int_{0}^T \int_{\Omega^t \backslash \Gamma^t} c^\epsilon \phi \,\d \bx \d t  
			= \int_{0}^T \int_{\Omega^t \backslash \Gamma^t} c \phi \,\d \bx \d t, 
		\end{equation*}
		for all $ \phi \in \cD(Q^T \backslash S^T) $, $ \phi \geq 0 $, one concludes that $ c \geq 0 $, a.e. in $ Q^T \backslash S^T $. The positivity of $ \hcss $ and $ \hg $ then follows automatically, as showed in \cite{AL2021a}, which completes the proof.
	\end{proof}
	

\appendix

\section{Stokes resolvent problem}
	\label{sec:sta-Stokes-lower-regularity} 
	In this section, we give a short proof the solvability of the following Stokes resolvent problem with mixed boundary conditions. Let $ \Omega \subset \bbr^3 $ be a bounded domain of class $ C^{3-} $ with boundary $\partial \OM = \Gamma_1 \cup \Gamma_2 $ consisting of two closed, disjoint, nonempty components. Consider the resolvent problem
	\begin{equation}
		\label{Eqs:sta-lambda}
		\begin{alignedat}{3}
			\lambda \bu - \Div (D^2 W(\bbi) \nabla \bu) + \nabla \pi & = \mathbf{f}, && \tin \OM, \\
			\Div \bu & = 0, && \tin \OM, \\
			\bu & = 0, && \ton \Gamma_1, \\
			(D^2 W(\bbi) \nabla \bu - \pi\bbi) \bn & = 0, && \ton \Gamma_2,
		\end{alignedat}
	\end{equation}
	where $ \bn $ is the outer unit normal on the boundary, $ W : \bbr^{3 \times 3} \rightarrow \bbr_+ $ is a scalar function with Assumption \ref{assumption:W-energy-density} holding. 
	\begin{theorem} \label{thm:sta-stokes-lambda}
		Let $ 1 < q < \infty $. Assume that $ \OM \subset \bbr^3 $ is the domain defined above.
		Given $ \mathbf{f} \in \Lq{q}(\OM)^3 $, there exists some $ \lambda_0 \in \bbr $ such that for all $ \lambda > \lambda_0 $, \eqref{Eqs:sta-lambda} admits a unique solution $ (\bu, \pi) $ satisfying
		\begin{equation*}
			\bu \in \W{2}(\OM)^3, \quad
			\pi \in \W{1}(\OM).
		\end{equation*} 
		Moreover, 
		\begin{equation*}
			\lambda \norm{\bu}_{\Lq{2}(\OM)^3} + \norm{\bu}_{\W{2}(\OM)^3} + \norm{\pi}_{\W{1}(\OM)} 
				\leq C 
					\norm{\mathbf{f}}_{\Lq{q}(\OM)^3}
				.
		\end{equation*}
	\end{theorem}
	\begin{proof}
		The proof is based on the maximal regularity of a generalized Stokes equation, see e.g. Bothe--Pr\"uss \cite[Theorem 4.1]{BP2007}, Pr\"uss--Simonett \cite[Theorem 7.3.1]{PS2016}. Let us recall the definition of solenoidal space
		\begin{equation*}
			\Lqs{q}(\OM) := \{\bu \in \Lq{q}(\OM)^3: \Div \bu = 0, \rv{\bn \cdot \bu}_{\Gamma_1} = 0 \}.
		\end{equation*}
		Then we define a Stokes-type operator as in Section \ref{sec:sta-Stokes-mixed} that
		\begin{equation*}
			\cA_q(\bu) := \bbp_q \big( - \Div (D^2 W(\bbi) \nabla \bu) \big) \text{ for all } \bu \in \cD(\cA_q),
		\end{equation*}
		with
		\begin{equation*}
			\cD(\cA_q) = \left\{ \bu \in \W{2}(\OM)^3 \cap \Lqs{q}(\OM) : \rv{\bu}_{\Gamma_1} = 0, \ \rv{\cP_{\bn}((D^2 W(\bbi) \nabla \bu) \bn)}_{\Gamma_2} = 0 \right\},
		\end{equation*}
		where $ \bbp_q $ denotes the \textit{Helmholtz--Weyl projection} on $ \Lqs{q}(\OM) $, see e.g. \cite[Appendix A]{Abels2010} for the existence of the projection with mixed boundary conditions. $ \cP_{\bn} := \bbi - \bn \otimes \bn $ is the tangential projection onto $ \partial \Omega $. As in Remark \ref{remark:epllipticity}, the operator $ - \Div (D^2 W(\bbi) \nabla \cdot) $ is strongly normally elliptic. By e.g. Pr\"uss--Simonett \cite[Theorem 7.3.2]{PS2016}, one knows that $ \lambda + \cA_q \in \cM \cR_q(\Lqs{q}(\OM)) $ for all $ \lambda > \lambda_0 := s(- \cA_q) $, where $ \cM \cR_q(\Lqs{q}(\OM)) $ means the class of maximal $ L^q $-regularity in $ \Lqs{q}(\OM) $ and $ s(- \cA_q) $ denotes the spectral bound of $ - \cA_q $. Consequently, for $ \lambda > \lambda_0 $ and $ \mathbf{f} \in \Lqs{q}(\OM) $, 
		\begin{equation*} 
			\lambda \bu + \cA_q \bu = \mathbf{f},
		\end{equation*}
		is uniquely solvable with $ \bu \in \cD(\cA_q) $, which implies that of \eqref{Eqs:sta-lambda}. 
%
		
		Now it remains to recover the pressure $ \pi $. To this end, we solve the Dirichlet--Neumann problem 
		\begin{equation}
			\label{Eqs:Laplace-pi}
			\begin{alignedat}{3}
				\Delta \pi & = \Div ( \mathbf{f} - \lambda \bu + \Div (D^2 W(\bbi) \nabla \bu) ), && \tin \OM, \\
				\ptial{\bn} \pi & = ( \mathbf{f} - \lambda \bu + \Div (D^2 W(\bbi) \nabla \bu) ) \cdot \bn, && \ton \Gamma_1, \\
				\pi & = (D^2 W(\bbi) \nabla \bu) \bn \cdot \bn, && \ton \Gamma_2,
			\end{alignedat}
		\end{equation}
		weakly, which is equivalent to the following weak formation
		\begin{equation}
			\label{Eqs:Laplace-pi-weak}
			\int_{\OM} \nabla \pi \cdot \nabla \varphi \,\d x = \int_{\OM} \underbrace{(\mathbf{f} - \lambda \bu + \Div (D^2 W(\bbi) \nabla \bu))}_{=: \tilde{\mathbf{f}}} \cdot \nabla \varphi \,\d x, \quad \forall \varphi \in W_{q',\Gamma_2}^1 (\OM).
		\end{equation}
		Since $ \mathbf{f} \in \Lq{q}(\OM)^3 $ and $ \bu \in \W{2}(\OM)^3 \cap \Lqs{q}(\OM) $, we have $ \tilde{\mathbf{f}} \in \Lq{q}(\OM)^3 $. Then \eqref{Eqs:Laplace-pi-weak} admits a unique solution $ \pi \in \W{1}(\OM) $, with the aid of \cite[Theorem 7.4.3]{PS2016}. The boundary regularity is easy due to the third equation of \eqref{Eqs:Laplace-pi}, combining the trace theorem.
	\end{proof}
	\begin{remark}
		\label{remark:lambda=0}
		As $ \cD(\cA_{q}) $ embeds compactly into $ \Lqs{q}(\OM) $, the Stokes-type operator $ \cA_q $ has compact resolvent. Therefore, its spectrum consists only of eigenvalues of finite algebraic multiplicity by spectral theory for compact operators (see e.g. Alt \cite{Alt2016}), and is independent of $ q $ by Sobolev embeddings. So it is enough to investigate these eigenvalues for the case $ q = 2 $. Let $ \omega $ be the eigenvalue of $ - \cA_2 $. Employing the energy method, Lemma \ref{remark:epllipticity} and the Korn's inequality, we have
		\begin{align*}
			\omega \int_{\OM} \abs{\bu}^2 \,\d x 
				= - \inner{\cA_2 \bu}{\bu}_{\Lqs{2}(\OM)}
				& = - \int_{\OM} D^2 W(\bbi) \nabla \bu : \nabla \bu \,\d x \\
				& \leq - C \int_{\OM} \abs{\nabla \bu + \nabla^\top \bu}^2 \,\d x
				\leq - C \int_{\OM} \abs{\nabla \bu}^2 \,\d x,
		\end{align*}
		which shows that $ \omega $ is real and nonpositive. Since $ \OM $ is bounded and $ \Gamma_1 $ is assumed to be nonempty, the Poincar\'e's inequality is valid and one obtains
		\begin{equation*}
			\omega \int_{\OM} \abs{\bu}^2 \,\d x 
				\leq - C \int_{\OM} \abs{\bu}^2 \,\d x,
		\end{equation*}
		which implies $ w \leq - C $, where $ C > 0 $ does not depend on $ \lambda $. Then one conclude that $ \lambda_0 $ defined in Theorem \ref{thm:sta-stokes-lambda} is negative and hence if $ \lambda = 0 $, the Theorem still holds true.
	\end{remark}

\end{document}